\documentclass{amsart}

\usepackage[english]{babel}
\usepackage{amsfonts, amsmath, amsthm, amssymb,amscd,indentfirst,mathrsfs}

\usepackage{color}
\usepackage{xcolor}
\usepackage{MnSymbol}
\usepackage{tikz}
\usepackage{esint}

\usepackage[utf8]{inputenc} 
\usepackage[T1]{fontenc}

\usepackage{graphicx}
\usepackage{epstopdf}
\usepackage{latexsym}

\usepackage{palatino}
\usepackage{marginnote}

\newtheorem{theorem}{Theorem}[section]

\newtheorem{proposition}[theorem]{Proposition}
\newtheorem{lemma}[theorem]{Lemma}
\newtheorem{definition}[theorem]{Definition}

\newtheorem{corollary}[theorem]{Corollary}

\newtheorem{remark}[theorem]{Remark}

\def\R{\mathbb{R}}

\def\d{\partial}

\def\bp{\begin{proof}}
\def\ep{\end{proof}}

\def\R{{\cal R}}

\newcommand{\tr}{{\rm tr \, }}

\def\R{\mathbb{R}}

\def\d{\partial}

\newcommand{\II}{\mathrm{II}}
\newcommand{\disp}{\displaystyle}
\newcommand{\loc}{\mathrm{loc}}

\newcommand{\metric}{\langle \, , \, \rangle}

\begin{document}

\title[Spacetime positive mass theorems]{Spacetime positive mass theorems for initial data sets with noncompact boundary}

\author{S\'{e}rgio Almaraz}
\address{Universidade Federal Fluminense (UFF), Instituto de Matem\'{a}tica, Campus do Gragoat\'a\\
Rua Prof. Marcos Waldemar de Freitas, s/n, bloco H, 24210-201, Niter\'{o}i, RJ, Brazil.}
              \email{sergio.m.almaraz@gmail.com}
\author{Levi Lopes de Lima}
\address{Universidade Federal do Cear\'a (UFC),
Departamento de Matem\'{a}tica, Campus do Pici, Av. Humberto Monte, s/n, Bloco 914, 60455-760,
Fortaleza, CE, Brazil.}
\email{levi@mat.ufc.br}
\author{Luciano Mari}
\address{Universit\'a degli Studi di Torino, Dipartimento di Matematica ``G. Peano", Via Carlo Alberto 10, 10123 Torino, Italy.}
\email{luciano.mari@unito.it}
\thanks{S. Almaraz has been partially suported by  CNPq/Brazil grant 
	309007/2016-0 and CAPES/Brazil grant 88881.169802/2018-01, and L. de Lima has been partially supported by CNPq/Brazil grant
	311258/2014-0. Both authors have been partially suported by FUNCAP/CNPq/PRONEX grant 00068.01.00/15. L. Mari is supported by the project SNS17\_B\_MARI by the Scuola Normale Superiore.}

\begin{abstract}
In this paper, we define an energy-momentum vector at the spatial infinity of either asymptotically flat or asymptotically hyperbolic initial data sets carrying a noncompact boundary. Under suitable dominant energy conditions imposed both on the interior and along the boundary, we prove the corresponding positive mass inequalities under the assumption that the underlying manifold is spin. 
\end{abstract}

\maketitle
\tableofcontents


\begin{abstract}
In this paper, we define an energy-momentum vector at the spatial infinity of either asymptotically flat or asymptotically hyperbolic initial data sets carrying a non-compact boundary. Under suitable dominant energy conditions (DECs) imposed both on the interior and along the boundary, we prove the corresponding positive mass inequalities under the assumption that the underlying manifold is spin. In the asymptotically flat case, we also prove a rigidity statement when the energy-momentum vector is light-like. 
Our treatment aims to underline both the common features and the differences between the asymptotically Euclidean and hyperbolic settings, especially regarding the boundary DECs.
\end{abstract}

\maketitle
\tableofcontents


\section{Introduction}

In General Relativity, positive mass theorems comprise the statement that, under suitable physically motivated energy conditions,  the total mass of an isolated gravitational system, as measured at its spatial infinity, is non-negative and vanishes only in case the corresponding initial data set propagates in time to generate the Minkowski space.  After the seminal contributions by Schoen-Yau \cite{SY1, SY2, SY3} and Witten \cite{Wi}, who covered various important cases, the subject has blossomed in a fascinating area of research; see  \cite{Bar,PT, BC, CM, XD, Ei, EHLS, SY4, Lo, HL} for a sample of relevant contributions in the asymptotically flat setting. 
More recently,  inspired by potential applications to the Yamabe problem on manifolds with boundary, a variant of the classical positive mass theorem for time-symmetric initial data sets carrying a non-compact boundary has been established in \cite{ABdL},  
under the assumption that the double of the underlying
manifold satisfies the standard (i.e. boundaryless) mass inequality. Hence, in view of the
recent progress due to Schoen-Yau \cite{SY4} and Lohkamp \cite{Lo}, the positive
mass theorem in [ABdL] actually holds in full generality. We also note that an alternative approach to the main result in \cite{ABdL},  based on the theory of free boundary minimal hypersurfaces and hence only suited for low dimensions, is presented in \cite{Ch}.   

Partly motivated by the so-called AdS/CFT correspondence in Quantum Gravity, 
there has been much interest in proving similar results in case the Minkowskian background is 
replaced by the anti-de Sitter spacetime. After preliminary contributions in \cite{M-O,AD}, the time-symmetric version has been established in \cite{Wa,CH} under the spin assumption. 
We also refer to \cite{Ma,CMT} for a treatment of the non-time-symmetric case, again in the spin context. Regarding
the not necessarily spin case, we should mention the results in \cite{ACG, CD, HJM}.
Notice that in this asymptotically hyperbolic setting, the time-symmetric spin case in the presence of a non-compact boundary appears in \cite{AdL}. As a consequence, a  rigidity result in the conformally compact Einstein (c.c.E.) setting was proved there, thus extending a celebrated achievement in \cite{AD}, which has interesting applications to the classical AdS/CFT correspondence. By its turn, the rigidity result in \cite[Theorem 1.3]{AdL}, which holds by merely assuming that the inner boundary is {\em mean convex},  	
may be of some interest in connection with recent developments involving the construction of a holograph dual of a conformal field theory defined on a manifold with boundary, the so-called AdS/BCFT correspondence \cite{T, CMG, AS}, where the problem of determining the structure of the moduli space of c.c.E. manifolds with a given conformal infinity  and having a {\em minimal} inner boundary  plays a key role.

The interesting applications of the purely Riemannian positive mass  theorems in  \cite{ABdL,AdL} mentioned above motivate their extension to the spacetime, non-time-symmetric case. The purpose of this paper is precisely to carry out this formulation in general and to establish the corresponding mass inequalities under the assumption that the manifold underlying the given initial data set is spin; see  Theorems \ref{main} and \ref{maintheohyp} below. 
For this, we adapt the well-known Witten's spinorial method which, in each case, provides a formula for the energy-momentum vector in terms of a spinor suitably determined by means of boundary conditions imposed both at infinity and along the non-compact boundary. In particular, since the computations in \cite{ABdL,AdL} are used in the present setting, our choice of boundary conditions along the non-compact boundary in each case matches those already considered there, namely, we employ MIT bag (respectively, chirality) conditions in the asymptotically flat (respectively, asymptotically hyperbolic) case; see also Remark \ref{disc:bound:cond} in this regard. 
We remark that  extensive discussions of these boundary conditions for Dirac-type operators and their Dirac Laplacians may be found in \cite{AE, Es, Ra, Gi, dL1,dL2}.
We also emphasize that a key step in our approach is the selection of  suitable dominant energy conditions along the non-compact boundary which constitute natural extensions of the mean convexity assumption adopted in \cite{ABdL,AdL}. In fact, the search for this kind of energy condition  was one of the motivations we had to pursue the investigations reported here. In any case, it is quite fortunate that these energy conditions and the boundary conditions on spinors mentioned above 
fit together so as to allow for a somewhat unified approach to our main results.

Although we have been able to establish positive mass inequalities in full generality for initial data sets whose underlying manifolds are spin, a natural question that arises is whether this topological assumption may be removed.
In the asymptotically flat case, one possible approach to this goal is to adapt, in the presence of the non-compact boundary, the classical technique based on MOTS (marginally outer trapped surfaces). Another promising  strategy is to proceed in the spirit of the time-symmetric case treated in \cite{ABdL} and  improve the asymptotics in order to apply the standard positive mass inequality to the ``double'' of the given initial data set. We hope to address those questions elsewhere.

Now we briefly describe the content of this paper. Our main results are Theorems \ref{main} and \ref{maintheohyp} which are proved in Sections \ref{proofmain} and \ref{secthyp}, respectively. 
These are rather straightforward consequences of the Witten-type formulae in Theorems \ref{genwit} and \ref{witt-type}, whose proofs make use of the material on spinors and the Dirac-Witten operator presented in Section \ref{spinors}. 
We point out that, although the spinor argument only provides rigidity assuming the vanishing of the energy-momentum vector, using a doubling argument we are able to reduce the general case to that one by applying the results in \cite{HL}.
Sections \ref{state} and \ref{enermom} are of an introductory nature, as they contain the asymptotic definition of  the  energy-momentum vectors and a proof that these objects are indeed geometric invariants of the given initial data set. We also include in Section \ref{state} a motivation for the adopted dominant energy conditions along the non-compact boundary which makes contact with the Hamiltonian formulation of General Relativity.

\section{Statement of results}\label{state}

Let $n\geq 3$ be an integer and consider  $(\overline M^{n+1},\overline g)$, an oriented and time-oriented $(n+1)$-dimensional Lorentzian manifold carrying a non-compact, timelike boundary $\overline \Sigma$. We assume that $\overline M$ carries a spacelike hypersurface $M$ with non-compact boundary $\Sigma=\overline \Sigma\cap M$. Also, we suppose that $M$ meets $\overline\Sigma$ {\em orthogonally} along $\Sigma$; see Remark \ref{orthoassrem} below.
Let $g=\overline g|_M$ be the induced metric and $h$ be the second fundamental form of the embedding $M\hookrightarrow\overline M$ with respect to the time-like, future directed unit normal vector field {$\bf n$} along $M$. 
As costumary, we assume that $\overline g$ is determined by extremizing the standard  Gibbons-Hawking action
\begin{equation}\label{gibbons}
\overline g\mapsto \int_{\overline M}\left(R_{\overline g}-2\Lambda+ \mathscr T\right)d\overline M+2\int_{\overline \Sigma}\left(H_{\overline g}+\mathscr S\right)d\overline \Sigma.
\end{equation}
Here, $R_{\overline g}$ is the scalar curvature of $\overline g$, $\Lambda\leq 0$ is the cosmological constant,  ${\II}_{\overline g}$ is the second fundamental form of $\overline{\Sigma}$ in the direction pointing towards $\overline{M}$, and ${H}_{\overline g} = \tr_{\overline{g}} {\II}_{\overline g}$ is its mean curvature. As usual, we have added to the purely gravitational action 
the stress-energy densities describing the non-gravitational contributions which are {\em independently} prescribed both in the interior of $\overline M$ ($\mathscr T$) and along the boundary $\overline \Sigma$ ($\mathscr S$).
In the following, we often consider an orthornormal frame $\{e_\alpha\}_{\alpha=0}^n$ along $M$ which is adapted to the embedding $M\hookrightarrow\overline M$ in the sense that $e_0={\bf  n}$. We work with the index ranges 
$$
0\leq \alpha,\beta,\cdots\leq n, \qquad 1\leq i,j,\cdots\leq n, \qquad 1 \leq A, B, \cdots \leq n-1, \quad 0\leq a,b,\cdots\leq n-1,
$$
and the components of the second fundamental form $h$ of $M$ in the frame $\{e_i\}$ are defined by
\[
h_{ij}   =\overline g ( \overline{\nabla}_{e_i} e_0, e_j),
\]
where $\overline \nabla$ is the Levi-Civita connection of $\overline g$. Along $\Sigma$, we also assume that the frame is adapted in the sense that $e_n=\varrho$, where $\varrho$ is the inward unit normal to $\Sigma$, so that $\{e_A\} \subset T\Sigma$.

In order to establish positive mass theorems, physical reasoning demands that the initial data set $(M,g,h,\Sigma)$ should satisfy suitable dominant energy conditions (DECs). In the interior of $M$, this is achieved in the usual manner, namely, we consider the {\em interior constraint map}
$$
\Psi_{\Lambda}(g,h)=2\left(\rho_{\Lambda}(g,h),J(g,h)\right),
$$ 
where 
\[
\rho_{\Lambda}(g,h)=\frac{1}{2}\left(R_g-2\Lambda-|h|_g^2+({\rm tr}_gh)^2\right), \quad J(g,h)={\rm div}_gh-d{\rm tr}_gh
\]
and $R_g$ is the scalar curvature of $g$. 

\begin{definition}\label{intdec}
	We say that $(M,g,h)$ satisfies the {\em interior DEC} if 
	\begin{equation}\label{intdec2}
	\rho_\Lambda\geq|J|
	\end{equation} 
	everywhere along $M$.
	\end{definition}

As we shall see below, prescribing DECs along $\Sigma$ is a subtler matter. 
In the time-symmetric case, which by definition means that $h=0$, the mass inequalities obtained in \cite{ABdL,AdL} confirm that mean convexity of $\Sigma$ (that is, $H_g\geq 0$, where $H_g$ is the mean curvature of $\Sigma\hookrightarrow M$ with respect to the inward pointing unit normal vector field $\varrho$) qualifies as the right boundary DEC. In analogy with (\ref{intdec2}), this clearly suggests that, in the non-time-symmetric case considered here, the appropriate boundary DEC should be expressed by a pointwise lower bound for $H_g$ in terms of the norm of a vector quantity constructed out of the geometry along $\Sigma$, which should vanish whenever $h=0$. However, a possible source of confusion in devising this condition is that the momentum component of the energy-momentum vector, appearing in the positive mass theorems presented below, possesses a manifestly distinct nature depending on whether it comes from asymptotically translational isometries tangent to the boundary if $\Lambda=0$, or asymptotically rotational isometries normal to the boundary if $\Lambda <0$; see Remark \ref{rem11} below.  Despite this difficulty, a reasonably unified approach may be achieved if, for the sake of motivation, we  appeal to the so-called Hamiltonian formulation of General Relativity. Recall that, in this setting, the spacetime $(\overline M,\overline g)$ is constructed by  infinitesimally deforming the initial data set $(M,g, h,\Sigma)$  in a transversal, time-like direction with speed $\partial_t=V{\bf n}+W^i{e_i}$, where  $V$ is the {lapse function} and $W$ is the {shift vector}. 
In terms of these quantities, and since $M$ is supposed to meet $\overline \Sigma$ orthogonally along 
$\Sigma$, the purely gravitational contribution $\mathscr{H}_{\rm grav}$ to the total Hamiltonian at each time slice is given by 
\begin{equation}\label{hamdens}
\frac{1}{2}\mathscr{H}_{\rm grav}(V,W)=\int_M\left(V\rho_\Lambda+W^iJ_i\right)dM
+\int_\Sigma \left(V H_g+W^i\left( \varrho\righthalfcup \pi\right)_i\right)d\Sigma, 
\end{equation} 
where $\pi:=h-({\rm tr}_gh)g$ is the {\rm conjugate momentum} (also known as the Newton tensor of $M\hookrightarrow \overline M$) and we assume for simplicity that $M$ is compact in order to avoid the appearance   of asymptotic terms in (\ref{hamdens}), which are not relevant for the present discussion. We refer to \cite{HH} for a direct derivation of this formula starting from the action (\ref{gibbons}); the original argument, which relies on the Hamilton-Jacobi method applied to (\ref{gibbons}), appears in \cite{BY}. We also mention that (\ref{hamdens}) may be derived in the framework of the formalism recently described in \cite{HW}, a sharpening of the celebrated Iyer-Wald  covariant phase space method \cite{IW} which properly takes into account the contributions coming from boundary terms.

Comparison of the interior and boundary integrands in (\ref{hamdens}) suggests the consideration of  the {\em boundary constraint map} 
\[
\Phi(g,h)=2(H_g,\varrho\righthalfcup \pi).
\]
The key observation now is that if we view $(V,W)$ as the infinitesimal generator of a symmetry yielding an energy-momentum charge, then the boundary integrand in (\ref{hamdens}) suggests that the corresponding DEC should somehow select the component of  
$\varrho\righthalfcup\pi$
aligned with $W$. 
In this regard, we note that $\varrho\righthalfcup\pi$ admits a tangential-normal decomposition with respect to the embedding $\Sigma\hookrightarrow M$, namely, 
\[
\varrho\righthalfcup\pi=\left((\varrho\righthalfcup\pi)^\downvdash,(\varrho\righthalfcup\pi)^\upvdash
\right)=\left(\pi_{nA},\pi_{nn}\right).
\]
It turns out that the boundary DECs employed here   
explore this natural decomposition. More precisely, as the lower bound for $H_g$ mentioned above we take the norm $|(\varrho\righthalfcup\pi)^\downvdash|$ of the tangential component if $\Lambda=0$ and the norm $|(\varrho\righthalfcup\pi)^\upvdash|$ of the normal component if $\Lambda<0$; see Definitions \ref{bddectan} and \ref{bddecnor} below.

\begin{remark}\label{orthoassrem}{\rm The orthogonality condition along $\Sigma=M\cap\overline\Sigma$  is rather natural from the viewpoint of the Hamilton-Jacobi analysis put forward in \cite{BY}. In fact, as argued there, it takes place for instance when we require that the corresponding Hamiltonian flow  evolves the initial data set $(M,g,h,\Sigma)$ in such a way that the canonical variables are not allowed to propagate accross  $\Sigma$.  We also remark 
that this assumption is automatically satisfied in case the spacetime enjoys the time-symmetry $t\mapsto -t$, where $t=0$ stands for $M$.
}
\end{remark}

For this first part of the discussion, which covers the asymptotically flat case, we  assume that $\Lambda=0$ in (\ref{gibbons}). To describe the corresponding reference spacetime,
let  $(\mathbb L^{n,1},\overline \delta)$ be the Minkowski space with coordinates $X=(x_0,x)$, $x=(x_1,\cdots,x_n)$,  endowed with the standard flat metric
\[
\langle X,X'\rangle_{\overline\delta}=-x_0x_0'+x_1x'_1+\cdots +x_nx_n'. 
\]
We denote by  $\mathbb L^{n,1}_+=\{X\in\mathbb L^{n,1};x_n\geq 0\}$ the {\em Minkowski half-space}, whose boundary $\partial \mathbb L^{n,1}_+$ is a time-like hypersurface. Notice that $\mathbb L^{n,1}_+$ carries the totally geodesic spacelike hypersurface $\mathbb R^n_+=\{x\in\mathbb L^{n,1}_+;x_0=0\}$ which is endowed with the standard Euclidean metric 
$\delta=\overline \delta|_{\mathbb R^{n}_+}$. Notice  that $\mathbb R^n_+$ also carries a totally geodesic boundary $\partial \mathbb R^n_+$.  
One aim of this paper is to formulate and prove, under suitable dominant energy conditions and in the spin setting, a positive mass theorem for spacetimes whose spatial infinity is modelled on the embedding $\mathbb R^n_+\hookrightarrow\mathbb L^{n,1}_+$.

We now make precise the requirement that the spatial infinity of $\overline M$, as observed along the initial data set $(M,g,h,\Sigma)$, is modelled on the inclusion $\mathbb R^n_+\hookrightarrow \mathbb L^{n,1}_+$.  
For large $r_0>0$ set $\mathbb R^n_{+,r_0}=\{x\in\mathbb R_+^n;|x|>r_0\}$, where  $|x|=\sqrt{x_1^2+...+x_n^2}$.

\begin{definition}\label{def:as:hyp}
	We say that $(M,g,h,\Sigma)$ is {\em{asymptotically flat}} (with a non-compact  boundary $\Sigma$) if there exist $r_0>0$, a region $M_{{\rm ext}}\subset M$, with $M\backslash M_{\rm{ext}}$ compact, and a diffeomorphism
	$$
	F:\mathbb R_{+,r_0}^n\to  M_{{\rm ext}}
	$$
	satisfying the following: 
	\begin{enumerate}
		\item as $|x|\to+\infty$, 
		\begin{equation}\label{asympthyp}
		|f|_\delta+|x||\partial f|_\delta+|x|^2|\partial^2f|_\delta=O(|x|^{-\tau}), 
		\end{equation}
		and 
		\[
		|h|_\delta+|x||\partial h|_{\delta}=O(|x|^{-\tau-1}),
		\]
		where $\tau>(n-2)/2$, $f:=g-\delta$, and we have identified $g$ and $h$ with their pull-backs under $F$ for simplicity of notation;
		\item there holds 
		\begin{equation}\label{intcond}
		\int_M|\Psi_0(g,h)|dM+\int_\Sigma|\Phi^\downvdash(g,h)|d\Sigma<+\infty,
		\end{equation}
		where 
		\begin{equation}\label{consts}
		\Phi^\downvdash(g,h)=2\left(
		\begin{array}{c}
		H_g\\
		{{(\varrho \righthalfcup \pi)^\downvdash}}
		\end{array}
		\right).
		\end{equation}
	\end{enumerate}
\end{definition}

\begin{figure}[t]
	\begin{center}
		\begin{tikzpicture}
		\begin{scope}[shift={(-6,0)}, scale=0.7]
		\draw (3,0) arc (180:0:3cm);
		\draw (0,0) -- (3,0); 
		\draw (3,0) -- (9,0);
		\draw (9,0) -- (12,0); 
		\draw [fill=black] (3,0) circle (2pt);
		\draw [fill=black] (9,0) circle (2pt);
		\draw [fill=black] (6-3*1.41/2,3*1.41/2) circle (1pt);
		\draw [->, very thick] (6-3*1.41/2,3*1.41/2) to (6-3,3);
		\draw [->, very thick] (9,0) to (10.3,0);
		\node at(9.8,-0.5) {$\vartheta$};
		\node at(3,2.2) {$\mu$};
		\node at(3,-0.5) {$S^{n-2}_r$};
		\node at(7,1.5) {$\Omega_r$};
		\node at(8,3) {$S^{n-1}_{r,+}$};
		\node at(1,3) {$M$};
		\node at(11,0.4) {$\Sigma$};
		\node at(6,-0.5) {$\Sigma_r$};
		\end{scope}
		\end{tikzpicture}
	\end{center}
	\caption{An initial data set with non-compact boundary.}
	\label{fig1}
\end{figure}
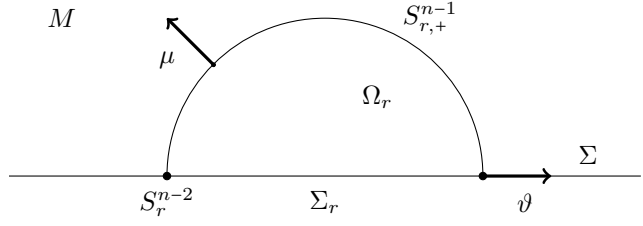


Under these conditions, we may assign to $(M,g,h,\Sigma)$ an energy-momentum-type asymptotic invariant as follows. Denote by $S^{n-1}_{r,+}$ the upper hemisphere of radius $r$ in the asymptotic region, $\mu$ its outward unit normal vector field (computed with respect to $\delta$), $S^{n-2}_r=\partial S^{n-1}_{r,+}$ and $\vartheta=\mu|_{S^{n-2}_r}$ its outward co-normal unit vector field (also computed with respect to $\delta$); see Figure \ref{fig1}.  

\begin{definition}\label{massdef}
	Under the conditions of Definition \ref{def:as:hyp}, the {\em energy-momentum} vector of the initial data set $(M,g,h,\Sigma)$ is the $n$-vector $(E,P)$ given by 
	\begin{equation}\label{enerdef}
	E=\lim_{r\to +\infty}\left[\int_{S^{n-1}_{r,+}}\left({\rm div}_\delta f-d{\rm tr}_\delta f\right)\left(\mu\right) dS^{n-1}_{r,+} + \int_{S^{n-2}_r}f\left(\partial_{x_n},\vartheta\right) dS^{n-2}_r\right],
	\end{equation} 
	and
	\begin{equation}\label{momdef}
	P_A=\lim_{r\to +\infty}2\int_{S^{n-1}_{r,+}}\pi\left(\partial_{x_A}, \mu\right) dS^{n-1}_{r,+},\:\:\:\:\:A=1,...,n-1.
	\end{equation}
	\end{definition}

If a chart at infinity $F$ as above is fixed, the 
energy-momentum vector $(E,P)$ may be viewed as a linear functional on the vector space 
$\mathbb R\oplus \mathfrak K^+_{\delta}$, where  elements in the first factor are identified with time-like translations normal to $\mathbb R^n_+\hookrightarrow \mathbb L_+^{n,1}$ and 
\begin{equation}\label{destkill}
\mathfrak K^+_{\delta}=\{W=\sum_{A=1}^{n-1}a_A\partial_{x_A};a_A\in\mathbb R \}
\end{equation} 
corresponds to 
translational Killing vector fields on $\mathbb R^n_+$ which are tangent to $\partial\mathbb R^n_+$.
Under a change of chart, it will be proved that $(E,P)$ is well defined (up to composition with an element of ${\rm SO}_{n-1,1}$); see Corollary \ref{welldef} below.
Thus, we may view $(E,P)$ as an element of the Minkowski space $\mathbb L^{n-1,1}$ at spatial infinity. Theorem \ref{main} below determines the causal character of this vector under suitable dominant energy conditions, showing that it is future-directed and causal in case the  manifold underlying the initial data set is spin. 

\begin{definition}\label{bddectan}
	We say that $(M,g,h,\Sigma)$ satisfies the {\em tangential boundary DEC} if there holds
	\begin{equation}\label{bddectan2}
	H_g\geq |(\varrho\righthalfcup \pi)^\downvdash|
	\end{equation}
	everywhere along $\Sigma$. 
	\end{definition}

We may now state our main result in the asymptotically flat setting (for the definition of the weighted H\"older spaces $C^{k,\alpha}_{-\tau}$, we refer the reader to \cite{HL}).

\begin{theorem}\label{main}
	Let $(M,g,h,\Sigma)$ be an asymptotically flat initial data set satisfying the DECs (\ref{intdec2}) and (\ref{bddectan2}). Assume that $M$ is spin. Then 
	\[
	E\geq |P|.
	\] 
	Furthermore, if $E = |P|$ and 
	\begin{equation}\label{decay_rhoJ_forrigid}
	\begin{array}{c}
	g-\delta \in C^{2,\alpha}_{-\tau}(M), \ \  h \in C^{1,\alpha}_{-\tau-1}(M) \qquad \text{for some $\alpha \in (0,1)$,} \qquad 
	\alpha + \tau > n-2, \\[0.2cm]
	\rho_0, J \in C^{0,\alpha}_{-n-\varepsilon}(M), \qquad \text{for some $\varepsilon>0$,} 
	\end{array}
	\end{equation}
then $E = |P| = 0$ and $(M,g)$ can be isometrically embedded in $\mathbb L^{n,1}_+$ with $h$ as the induced second fundamental form. Moreover, $\Sigma$ is totally geodesic (as a hypersurface in $M$), lies on $\partial\mathbb L^{n,1}_+$ and $M$ is orthogonal to $\partial\mathbb L^{n,1}_+$ along $\Sigma$. In particular, $h_{nA}$ vanishes on $\Sigma$.
	\end{theorem}
	
In the physically relevant case $n=3$, the spin assumption poses no restriction whatsoever since any oriented $3$-manifold is spin. 
Theorem \ref{main} is the natural extension of Witten's celebrated result \cite{Wi, GHHP, PT, He1, D, XD} to our setting, and its time-symmetric version appears in \cite[Section 5.2]{ABdL}. The mean convexity condition $H_g\geq 0$, which plays a prominent role in \cite{ABdL}, is deduced here as an immediate consequence of the boundary DEC (\ref{bddectan2}), thus acquiring a justification on purely {\em physical} grounds;  see also the next remark.

\begin{remark}\label{declag}{\rm  The DECs (\ref{intdec2}) and (\ref{bddectan2}) admit a neat interpretation derived from the Lagrangian formulation. Indeed, after extremizing (\ref{gibbons}) we get the field equations 
		\begin{equation}\label{efe}
		\begin{cases}
		\reversemarginpar
		{\rm Ric}_{\overline g}-\frac{R_{\overline g}}{2}\overline g+\Lambda \overline g=T,&\text{in}\:\overline M,
		\\
		{\II}_{\overline g} -H_{\overline g} \overline g|_{\overline \Sigma}=S,&\text{on}\:\overline \Sigma.
		\end{cases}
		\end{equation}
		Here, ${{\rm Ric}_{\overline g}}$  is the Ricci tensor of $\overline g$ and $T$ is the stress-energy tensor in $\overline M$, which is locally given by  
		\[
		T_{\alpha\beta}=\frac{1}{\sqrt{-\det \overline g}}\frac{\partial}{\partial\overline g^{\alpha\beta}}\left(\sqrt{-\det \overline g}\mathscr T\right).
		\]
		Also, the boundary stress-energy tensor $S$ on $\overline \Sigma$ is similarly expressed in terms of $\mathscr S$ and $\overline g|_{\overline \Sigma}$.  
		It is well-known that restriction of the first system of equations in (\ref{efe}) to $M$ yields the {\em interior constraint equations}
		\begin{equation}\label{constraint1n}
		\left\{
		\begin{array}{rcl}
		\rho_\Lambda & = & T_{00},\\
		J_i & = & T_{0i},
		\end{array}
		\right.
		\end{equation}
		so that (\ref{intdec2}) is equivalent to saying that the vector $T_{0\alpha}$ is causal and future directed. On the other hand, if $\varrho=e_n$ and $\overline\varrho$ are the inward unit normal vectors to $\Sigma$ and $\overline\Sigma$, respectively, then the assumption that $M$ meets $\overline \Sigma$
		orthogonally along $\Sigma$ means that $\overline\varrho|_\Sigma=\varrho$ and $e_0$ is tangent to $\overline\Sigma$. We then have
		\begin{eqnarray*}
			S_{00} & = & (\Pi_{\overline g})_{00}+H_{\overline g} \: = \: (\Pi_{\overline g})_{AA}\\
			& = & \langle \overline\nabla_{e_A}e_A,\overline\varrho\rangle \: = \: \langle \nabla_{e_A}e_A,\varrho\rangle 
			\:=  H_g
		\end{eqnarray*}	
and
		\begin{eqnarray*}
			S_{0A} & = & (\Pi_{\overline g})_{0A}
			 =  \langle \overline\nabla_{e_A}e_0,\overline\varrho\rangle\\
			& = & \langle \overline\nabla_{e_A}e_0,\varrho\rangle
			 =  h(e_A,\varrho)\\
			& = &  (\varrho\righthalfcup h)_A
			 =  (\varrho\righthalfcup \pi)_A,
		\end{eqnarray*}
		where in the last step we used that $(\varrho\righthalfcup g)_A=0$. Thus, we conclude that the restriction of the second system of equations in (\ref{efe}) to $\Sigma$ gives the {\em boundary constraint equations}
		\begin{equation}\label{constraint2}
		\left\{
		\begin{array}{rcl}
		H_g & = & S_{00},\\
		(\varrho \righthalfcup {{\pi}})_A & = & S_{0A}.
		\end{array}
		\right.
		\end{equation}
		As a consequence, (\ref{bddectan2}) is equivalent to requiring that the vector $S_{0a}$ is causal and future directed. We note however that the boundary DEC in the asymptotically hyperbolic case discussed below does not seem to admit a similar interpretation coming from the Lagrangian formalism; see Remark \ref{rem22}. 
		}
	\end{remark}

\begin{remark}\label{rem-cauchy}\rm Similarly to what happpens in (\ref{constraint1n}), the boundary constraint equations in (\ref{constraint2}) seem to relate initial values of fields along $\Sigma$ instead of determining how fields evolve in time. It remains to investigate the question of whether these boundary constraints play a role in the corresponding initial-boundary value Cauchy problem, so as to provide an alternative to the approaches in \cite{RS,KW}.
	\end{remark}

\begin{remark}
{\rm
One could try a connection with the positive mass theorem with corners in \cite[Section VI]{Sh} as follows.
Since the causal character of our energy-momentum vector of Definition 2.4 is invariant by the choice of asymptotically flat coordinates, we can choose one that satisfies $g_{An}=0$ on $\Sigma$ for $A=1,...,n-1$. In those coordinates, the integral on $S^{n-2}_r$ vanishes in (2.7). By a doubling argument, we obtain an asymptotically flat manifold 
$(\widetilde M,\widetilde g)$ in the usual sense, i.e., without boundary and modeled on $\mathbb R^n$ at infinity. Here, the metric $\widetilde g$, which is continuous but nonsmooth along $\Sigma$, is asymptotically flat in the sense of \cite{LL, Sh}. 
However, this construction does not lead to a causal corner in the sense of \cite[Definition 17]{Sh}. This is because our boundary DEC is weaker than the one in \cite{Sh} as it does not include the $n$-th direction (across $\Sigma$).
The point is that the norm of our linear-momentum vector already gives a geometric invariant without including this direction. Although the orthogonal direction may lead to some other geometric invariant, we believe that the first $n-1$ components are the minimum required to obtain such geometric invariance.
}
\end{remark}

Now we discuss the case of negative cosmological constant.
As already mentioned, a positive mass inequality for time-symmetric asymptotically hyperbolic initial data sets endowed with a non-compact boundary has been proved in \cite[Theorem 5.4]{AdL}. Here, we pursue this line of research one step further and present a spacetime version of this result. In particular, we recover the mean convexity assumption along the boundary as an immediate consequence of the suitable boundary DEC.

To proceed, we assume that the initial data set $(M,g,h,\Sigma)$ is induced by the embedding $(M,g)\hookrightarrow (\overline M,\overline g)$, where $\overline g$ extremizes (\ref{gibbons}) with 
$\Lambda=\Lambda_n:=-{n(n-1)}/{2}$.
Recall that,
using coordinates $Y=(y_0, y)$, $y=(y_1,\cdots,y_n)$, the {\em anti-de Sitter space}
is the spacetime $({\rm AdS}^{n,1}, \overline b)$, where 
\[
\overline b=-(1+|y|^2)dy_0^2+b, \quad b=(1+|y|^2)^{-1}d|y|^2+|y|^2h_0,
\]
$|y|=\sqrt{y_1^2+\cdots +y_n^2}$ and, as usual, $h_0$ is the standard metric on the unit sphere $\mathbb S^{n-1}$. Our reference  spacetime now  is the {\em AdS half-space} ${{\rm AdS}_+^{n,1}}$ defined by the requirement $y_n\geq 0$. Notice that this space carries a boundary $\partial {\rm AdS}^{n,1}_+=\{Y\in {\rm AdS}^{n,1}_+;y_n=0\}$ which is timelike and totally geodesic. Our aim is to formulate a positive mass inequality for spacetimes whose spatial infinity is  modelled  on the inclusion $\mathbb H_n^+\hookrightarrow {\rm AdS}^{n,1}_+$, where  $\mathbb H_+^n=\{Y\in {\rm AdS}^{n,1}_+;y_0=0\}$ is the totally geodesic spacelike slice which, as the notation suggests, can be identified with the hyperbolic half-space $(\mathbb H^n_+,b)$ appearing in \cite{AdL}.

We now make precise the requirement that the spatial infinity of $\overline M$, as observed along the initial data set $(M,g,h,\Sigma)$, is modelled on the inclusion $\mathbb H^n_+\hookrightarrow {\rm AdS}^{n,1}_+$.  
For all $r_0>0$ large enough let us set $\mathbb H^n_{+,r_0}=\{y\in\mathbb H_+^n; |y|> r_0\}$.   

\begin{definition}\label{def:as:hyp2}
	We say that the initial data set $(M,g,h,\Sigma)$ is {\em{asymptotically hyperbolic}} (with a non-compact  boundary $\Sigma$) if there exist $r_0>0$, a region $M_{{\rm ext}}\subset M$, with $M\backslash M_{\rm{ext}}$ compact, and a diffeomorphism
	$$
	F:\mathbb H_{+,r_0}^n\to  M_{{\rm ext}}, 
	$$
	satisfying the following: 
	\begin{enumerate}
		\item as $|y|\to+\infty$, 
		\begin{equation}\label{asympthyp2}
		|f|_b+|\nabla_b f|_b+|\nabla_b^2f|_b=O(|y|^{-\kappa}), 
		\end{equation}
		and 
		\[
		|h|_b+|\nabla_b h|_b=O(|y|^{-\kappa}),
		\]
		where $\kappa>n/2$, $f:=g-b$, and  we have identified $g$ and $h$ with their pull-backs under $F$ for simplicity of notation;
		\item there holds 
		\begin{equation}\label{intcond2}
		\int_M |y| |\Psi_{\Lambda_n}(g,h)|dM+\int_\Sigma |y| |\Phi^\upvdash(g,h)|d\Sigma<+\infty,
		\end{equation}
		where 
		\begin{equation}\label{consts2}
		\Phi^\upvdash(g,h)=2\left(
		\begin{array}{c}
		H_g\\
		{{(\varrho \righthalfcup \pi)^\upvdash}}
		\end{array}
		\right)
		\end{equation}
		and $|y|$ has been smoothly extended to $M$.
	\end{enumerate}
\end{definition}  

Under these conditions, we may assign to $(M,g,h,\Sigma)$ an energy-momentum-type asymptotic invariant as follows. We essentially keep the previous notation and denote by $S^{n-1}_{r,+}$ the upper hemisphere of radius $r$ in the asymptotic region, $\mu$ its outward unit vector field (computed with respect to $b$), $S^{n-2}_r=\partial S^{n-1}_{r,+}$ and $\vartheta=\mu|_{S^{n-2}_r}$ its outward co-normal unit vector field (also computed with respect to $b$). As in  \cite{AdL}, we consider the space of static potentials
\[
\mathcal N_b^+=\left\{
V:\mathbb H^n_+\to \mathbb R;\nabla_b^2V=Vb\,\,\,{\rm in}\,\,\,\mathbb H^n_+, \frac{\partial V}{\partial y_n}=0
\,\,\,{\rm on}\,\,\,\partial\mathbb H^n_+\right\}.
\]
Thus, $\mathcal N_b^+$ is generated by $\{V_{(a)}\}_{a=0}^{n-1}$, where $V_{(a)}=x_a|_{\mathbb H^n_+}$ and here we view $\mathbb H^n_+$ embedded as the upper half hyperboloid in $\mathbb L^{n,1}_+$ endowed with coordinates $\{x_\alpha\}$.  Notice that $V=O(|y|)$ as $|y|\to +\infty$ for any $V\in\mathcal N_b^+$. Finally, we denote by  $\mathfrak{Kill}(\mathbb H^n)$ ($\mathfrak{Kill}({\rm AdS}^{n,1})$, respectively) the Lie algebra of Killing vector fields of $\mathbb H^n$ (${\rm AdS}^{n,1}$, respectively) and set 
$\mathcal N_b=\mathcal N_b^+\oplus[x_n|_{\mathbb H^n}]$.
Note the isomorphism $\mathfrak{Kill}({\rm Ads}^{n,1})\cong \mathcal N_b\oplus \mathfrak{Kill}(\mathbb H^n)$, where
each $V\in\mathcal N_b$ is identified with the Killing vector field in ${\rm AdS}^{n,1}_+$ whose restriction to the spacelike slice $\mathbb H^n_+$ is $V(1+|y|^2)^{-1/2}\partial_{{y_0}}$.

\begin{definition}\label{massdef2}
	The {\em energy-momentum} of the asymptotically hyperbolic initial data set $(M,g,h,\Sigma)$ is the linear functional 
$$
\mathfrak m_{(g,h,F)}: \mathcal N_b^+\oplus\mathfrak{K}_b^+\to \mathbb R
$$ 
given by 
	\begin{eqnarray}\label{enerdef2}
	\mathfrak m_{(g,h,F)}(V,W) & = & \lim_{r\to +\infty}\left[\int_{S^{n-1}_{r,+}}\widetilde{\mathbb U}(V,f)(\mu) dS^{n-1}_{r,+}+
	\int_{S^{n-2}_{r}}Vf(\varrho_b,\vartheta)dS^{n-2}_{r}\right]\nonumber\\
	& & \quad +\lim_{r\to +\infty}2\int_{S^{n-1}_{r,+}}\pi(W,\mu)dS^{n-1}_{r,+}, 
	\end{eqnarray} 
	where $\varrho_b$ is the inward unit normal vector to $\partial\mathbb H^n_+$, 
	\begin{equation}\label{momdef2}
	\widetilde{\mathbb U}(V,f)=V({\rm div}_bf-d{\rm tr}_bf)-{\nabla_bV}\righthalfcup f+{\rm tr}_bf\, dV,
	\end{equation}
	and  $\mathfrak K^+_b$ is the subspace of elements of $\mathfrak{Kill}(\mathbb H^n)$ which are orthogonal to $\partial\mathbb H^n_+$.
\end{definition} 

\begin{remark}\label{rem11}
\rm {As already pointed out, the energy-momentum invariant in Definition \ref{massdef} may be viewed as a linear functional on the space of  {\em translational} Killing vector fields $\mathbb R\oplus \mathfrak K^+_\delta$; see the discussion surrounding  (\ref{destkill}). This should be contrasted to the Killing vector fields in the space $\mathcal N_b^+\oplus\mathfrak{K}_b^+$ appearing in Definition \ref{massdef2}, which are {\em rotational} in nature. Besides, the elements of $\mathfrak K^+_\delta$ are tangent to $\partial\mathbb R^n_+$ whereas those of $\mathfrak K^+_b$ are normal to $\partial\mathbb H^n_+$. Despite these notable distinctions between the associated asymptotic invariants, it is a remarkable feature of the spinorial approach that the corresponding mass inequalities can be established by quite similar methods.}
\end{remark}
\begin{definition}\label{bddecnor}
	We say that $(M,g,h,\Sigma)$ satisfies the {\em normal boundary DEC} if there holds
	\begin{equation}\label{bddecnor2}
	H_g\geq |(\varrho\righthalfcup \pi)^\upvdash|
	\end{equation}
	everywhere along $\Sigma$. 
\end{definition}
\begin{remark}\label{rem22}
{\rm Differently from what happens to the tangential boundary DEC in Definition \ref{bddectan}, the requirement in (\ref{bddecnor2}), which involves the normal component of $\varrho\righthalfcup \pi$, does not seem to admit an interpretation in terms of the Lagrangian formulation underlying the field equations (\ref{efe}), the reason being that only the variation of the tangential component of $\overline g$ shows up in the boundary contribution to the variational formula for the action (\ref{gibbons}). This distinctive aspect of the Lagrangian approach explains why the second system of equations in (\ref{efe}) is explicited solely in terms of tensorial quantities acting on $T\overline\Sigma$, which  leads to the argument in Remark \ref{declag} and eventually justifies the inclusion of the Hamiltonian motivation for the boundary DECs based on (\ref{hamdens}).}
\end{remark}

We now state our main result in the asymptotically hyperbolic  case.
This extends to our setting a previous result by Maerten \cite{Ma}.

\begin{theorem}\label{maintheohyp}
	Let $(M,g,h,\Sigma)$ be an asymptotically hyperbolic initial data set as above and assume that the DECs (\ref{intdec2}) and (\ref{bddecnor2}) hold. Assume further that $M$ is spin. Then there exists a quadratic map 
	\[
	\mathbb C^d\stackrel{\mathcal K}{\longrightarrow} \mathcal N_b^+\oplus \mathfrak{K}_b^+,
	\] 
	where $d=\big[\frac{n}{2}\big]$, such that the composition
	\[
	\mathbb C^d\stackrel{\mathcal K}{\longrightarrow} \mathcal N_b^+\oplus \mathfrak{K}_b^+\stackrel{\mathfrak{m}_{(g,h,F)}}{\longrightarrow}\mathbb R
	\]
	is a hermitean quadratic form $\widetilde{\mathcal K}$ satisfying $\widetilde{\mathcal K}\geq 0$. Also, if $\widetilde{\mathcal K}=0$ then $(M,g)$ is isometrically embedded in ${\rm AdS}^{n,1}_+$ with $h$ as the induced second fundamental form and, 
	besides, $\Sigma$ is totally geodesic (as a hypersurface in $M$), lies on $\partial{\rm AdS}^{n,1}_+$, $\overline\Sigma$ is geodesic (as a hypersurface of $\overline M$) in directions tangent to  $\Sigma$, and $h_{nA}$  vanishes on $\Sigma$.
\end{theorem}
\begin{remark}
	The space $\mathbb C^d$ should be viewed as the space of Killing spinors on $\mathbb H^n$; see the paragraph above Proposition \ref{killingder}.
\end{remark}

Differently from its counterpart in \cite{Ma}, the mass inequality $\widetilde{\mathcal K}\geq 0$ admits a nice geometric interpretation in any dimensions $n\geq 3$ as follows. 
Either $\mathcal N^+_b$ and $\mathfrak K^+_b$ can be canonically identified with $\mathbb L^{1,n-1}$ with its inner product 
\begin{equation}\label{lorinner}
\langle\langle z,w\rangle\rangle=z_0w_0-z_1w_1-\cdots- z_{n-1}w_{n-1}.
\end{equation}
The identification $\mathcal N^+_b\cong\mathbb L^{1,n-1}$ is done as in \cite{CH}  by regarding $\{V_{(a)}\}_{a=0}^{n-1}$ as an orthonormal basis and endowing $\mathcal N_b^+$   
with a time orientation by declaring $V_{(0)}$ as future directed. 
Then the isometry group of the totally geodesic spacelike slice $(\mathbb H^n_{+},b,\partial \mathbb H^n_{+})$, which is formed by those isometries of $\mathbb H^n$ preserving $\partial\mathbb H^n_+$, acts naturally on $\mathcal N_b^+$  in such a way  that the Lorentzian metric (\ref{lorinner}) is preserved (see \cite{AdL}).
On the other hand, as we shall see in Proposition \ref{propo:K+}, the identification $\mathfrak K^+_b\cong \mathbb L^{1,n-1}$ is obtained in a similar way.
Thus, in the presence of a chart $F$, the mass functional $\mathfrak m_{(g,h,F)}$ may be regarded as a {\em pair} of Lorentzian vectors $(\mathcal E, \mathcal P)\in \mathbb L^{1,n-1}\oplus \mathbb L^{1,n-1}$ (see (\ref{pair:vect})). 
In terms of this geometric interpretation of the mass functional, Theorem \ref{maintheohyp} may be rephrased as the next result, whose proof also appears in Section \ref{secthyp}. 

\begin{theorem}\label{maintheoohyp2}	
Under the conditions of Theorem \ref{maintheohyp},
the vectors $\mathcal E$ and $\mathcal P$, viewed as elements of $\mathbb L^{1,n-1}$, are both causal and  future directed. Moreover, if both of these vectors vanish then the rigidity statements in Theorem \ref{maintheohyp}
hold true. 
\end{theorem}

\begin{remark}\label{hyptimesym}{\rm 
    If the initial data set in Theorem \ref{maintheohyp} is {time-symmetric} then the mass functional reduces to a map $\mathfrak m_{(g,0,F)}:\mathcal N_b^+\to\mathbb R$. This is precisely the situation studied in \cite{AdL}. 
Under the corresponding DECs, it follows from \cite[Theorem 5.4]{AdL} that $\mathfrak m_{(g,0,F)}$, viewed as an element of $\mathcal N_b^+$,  is causal  and future directed. In other words, there holds the mass inequality $\langle\langle \mathfrak m_{(g,0,F)},\mathfrak m_{(g,0,F)}\rangle\rangle\geq 0$, which clearly is the time-symmetric version of the conclusion $\widetilde{\mathcal K}\geq 0$ in the broader setting of
	Theorem \ref{maintheohyp}.
	Moreover, if $\widetilde{\mathcal K}=0$ then $\mathfrak m_{(g,0,F)}$ vanishes and the argument in \cite[Theorem 5.4]{AdL} implies that $(M,g,\Sigma)$ is isometric to $(\mathbb H^n_+,b,\partial \mathbb H^n_+)$. We note however that in \cite{AdL} this same 
	 isometry is achieved just by assuming that  $\mathfrak m_{(g,0,F)}$ is null (that is, lies on the null cone associated to (\ref{lorinner})).
	Anyway, the rigidity statement in Theorem \ref{maintheohyp} implies that $(M,g,0,\Sigma)$ is isometrically embedded in ${\rm AdS}^{n,1}_+$. We then conclude that in the time-symmetric case  the assumption $\widetilde{\mathcal K}=0$ actually implies that the embedding $M\hookrightarrow\overline M$ reduces to the totally geodesic embedding $\mathbb H^n_+\hookrightarrow {\rm AdS}^{n,1}_+$. }  
\end{remark}


\section{The energy-momentum vectors}\label{enermom}

In this section we indicate how the asymptotic invariants considered in the previous section are well defined in the appropriate sense. 

Let $(M,g,h,\Sigma)$ be an initial data set either asymptotically flat, as in Definition \ref{def:as:hyp}, or asymptotically hyperbolic, as in Definition \ref{def:as:hyp2}. In the former case the model will be $(\mathbb R^n_+,\delta, 0, \partial\mathbb R^n_+)$, and in the latter it will be  $(\mathbb H^n_+,b, 0, \partial\mathbb H^n_+)$. 
We will denote these models by $(\mathbb E^n_+, g_0, 0,\partial\mathbb E^n_+)$, so that $\mathbb E^n_+$ stands for either $\mathbb R^n_+$ or  $\mathbb H^n_+$, and $g_0$ by either $\delta$ or $b$.
In particular, Definitions \ref{def:as:hyp} and \ref{def:as:hyp2} ensure that $f=g-g_0$ has appropriate decay order. Finally, we denote by $\varrho_{g_0}$ the inward unit normal vector to $\partial\mathbb E^n_+$.

We define $\mathcal N^+_{g_0}$ as a subspace of the vector space of solutions $V\in C^{\infty}(\mathbb E^n_+)$ to
\begin{equation}\label{eq:static:gen}
\begin{cases}
\nabla_{g_0}^2V-(\Delta_{g_0}V)g_0-V\text{Ric}_{g_0}=0 &\text{in}\:\mathbb E^n_+,
\\[0.2cm]
\displaystyle\frac{\partial V}{\partial\varrho_{g_0}}\gamma_0+V\Pi_{g_0}=0 &\text{on}\:\partial\mathbb E^n_+,
\end{cases}
\end{equation}
chosen as follows. Set $\mathcal N^+_b$ to be the full space itself and set $\mathcal N^+_{\delta}$ to be the space of constant functions. So, $\mathcal N^+_b\cong \mathbb L^{1,n-1}$ and  $\mathcal N^+_{\delta}\cong \mathbb R$.
\begin{remark}
\rm{
Although the second fundamental form $\Pi_{g_0}$ of $\partial\mathbb E^n_+$ vanishes for either $g_0=\delta$ or $g_0=b$, we will keep this term in this section in order to preserve the generality in our calculations.  
}
\end{remark}

We define $\mathfrak K^+_{g_0}$ as a subspace of the space of $g_0$-Killing vector fields as follows. Choose $\mathfrak K^+_b$ as the subset of all such vector fields of $\mathbb H^n$ which are orthogonal to $\partial\mathbb H^n_+$. The Killing fields $\{L_{an}\}_{a=0}^{n-1}$ displayed in the proof of Proposition \ref{propo:K+} below constitute a basis for  $\mathfrak K^+_b$; see also Remark \ref{lapse-shift}. Choose $\mathfrak K^+_{\delta}$ as the translations of $\mathbb R^n$ which are tangent to $\partial \mathbb R^n_+$, so that $\mathfrak K^+_{\delta}$ is generated by 
$\{\partial_{x_A}\}_{A=1}^{n-1}$.

Let $F$ be a chart at infinity for $(M,g,h,\Sigma)$. As before, we identify $g$ and $h$ with their pull-backs by $F$, and set $f=g-g_0$.

\begin{definition}\label{def:charge}
For $(V, W)\in \mathcal N^+_{g_0}\oplus\mathfrak{K}_{g_0}^+$ we define the {\rm{charge density}} 
\begin{align}\label{defU}
\mathbb U_{(f,h)}(V,W)=\,&V\Big( {\rm div}_{g_0}f-d{\rm tr}_{g_0}f \Big) - \nabla_{g_0} V \righthalfcup f + (\tr_{g_0} f) d V 
\\
&+ 2\big( {W}\righthalfcup h-({\rm tr}_{g_0}h)W_\flat\big),\notag
\end{align}
where $W_\flat=g_0(\cdot, W)$,
and the {\rm{energy-momentum functional}}
\begin{align}\label{limit}
\mathfrak m_{(g,h,F)}(V,W)=\lim_{r\to\infty}
\left\{
\int_{S_{r,+}^{n-1}}\mathbb U_{(f,h)}(V,W)(\mu)+\int_{S_{r}^{n-2}}Vf(\varrho_{g_0},\vartheta)
\right\}.
\end{align}
\end{definition}
\noindent
{\bf{Agreement.}} In this section we are omitting the volume elements in the integrals, which are all taken with respect to $g_0$.
\begin{proposition}\label{propo:mass:int}
The limit in (\ref{limit}) exists. In particular, it defines a linear functional
$$
\mathfrak m_{(g,h,F)}:\mathcal N_{g_0}^+\oplus \mathfrak{K}_{g_0}^+\to \mathbb R.
$$
If $\widetilde F$ is another assymptotic coordinate system for $(M,g,h,\Sigma)$ then 
\begin{equation}\label{propo:mass:int:1}
\mathfrak m_{(g,h,\widetilde F)}(V,W)=m_{(g,h,F)}(V\circ A, A^*W)
\end{equation}
for some isometry $A:\mathbb E^n_+\to \mathbb E^n_+$ of $g_0$.
\end{proposition}

Before proceeding to the proof of Proposition \ref{propo:mass:int}, we state two immediate corollaries, one for the asymptotically flat case and the other for the asymptotically hyperbolic one.

In the the asymptotically flat case, observe that the energy-momentum vector $(E,P)$ of Definition \ref{massdef} is given by
\begin{equation*}
\begin{array}{cc}
E=\mathfrak m_{(g,h, F)}(1,0), & P_A=\mathfrak m_{(g,h, F)}(0,\partial_{x_A}), \:\:\:A=1,...,n-1.
\end{array}
\end{equation*}
\begin{corollary}\label{welldef}
	If $(M,g,h,\Sigma)$ is an asymptotically flat initial data, then $(E,P)$  is well defined (up to composition with an element of ${\rm SO}_{n-1,1})$.
 In particular, the causal character of $(E,P)\in \mathbb L^{n-1,1}$ and the quantity 
	\[
	\langle(E,P),(E,P)\rangle=-E^2+P^2_1+...+P^2_{n-1}
	\] 
	do not depend on the chart $F$ at infinity chosen to compute $(E,P)$.
\end{corollary}

In the asymptotically hyperbolic case, observe that the functional $\mathfrak m_{(g,h,F)}$ coincides with the one of Definition \ref{massdef2}. In order to understand the space $\mathfrak K^+_b$, we state the following result:
\begin{proposition}\label{propo:K+}
If $\mathfrak {Kill}(\mathbb H^n)$ is the space of Killing vector fields on $\mathbb H^n$, then there are isomorphisms 
$$
\begin{array}{cc}
\mathfrak {Kill}(\mathbb H^n)\cong \mathfrak K_b^+\oplus \mathfrak{Kill}(\mathbb H^{n-1}), & \mathfrak K_b^+\cong \mathbb L^{1,n-1}.
\end{array}
$$
Moreover, the space $\text{Isom} (\mathbb H^n_+)$ of isometries of $\mathbb H^n_+$ acts on $\mathfrak {Kill}(\mathbb H^n)$ preserving the decomposition $\mathfrak K_b^+\oplus \mathfrak{Kill}(\mathbb H^{n-1})$. In particular, $\text{Isom} (\mathbb H^n_+)$ acts on $\mathfrak K^+_b$ by isometries of $\mathbb L^{1,n-1}$.
\end{proposition}
\begin{proof}
Observe that $\mathfrak {Kill}(\mathbb H^n)$ is generated by 
$$
\begin{array}{cc}
L_{0j}=x_0\partial_{x_j}+x_j\partial_{x_0}, & L_{ij}=x_i\partial_{x_j}-x_j\partial_{x_i},\: i<j,
\end{array}
$$
where $x_0,...,x_n$ are the coordinates of $\mathbb L^{n,1}$, and $\mathbb H^n\hookrightarrow \mathbb L^{n,1}$ is represented in the hyperboloid model. Restricting to $x_n=0$ we obtain
$$
L_{0B}\big|_{x_n=0}=x_0\partial_{x_{B}}+x_{B}\partial_{x_0}, \quad L_{AB}\big|_{x_n=0}=x_{A}\partial_{x_{B}}-x_{B}\partial_{x_{A}}, \quad  A<B,
$$
and 
\[
L_{0n}\big|_{x_n=0}=x_0\partial_{x_n}, \quad L_{An}\big|_{x_n=0}=x_A\partial_{x_n}, \quad A,B=1,...,n-1.
\]
This shows that $\{L_{an}|_{\mathbb H^n_+}\}_{a=0}^{n-1}$ is a basis for $\mathfrak K^+_b$. Moreover, since $V_{(a)}=x_a|_{\mathbb H^n_+}$, we obtain the isomorphisms $\mathfrak K_b^+\cong \mathcal N_b^+\cong\mathbb L^{1,n-1}$.
\end{proof}
	\begin{remark}\label{lapse-shift}
		{\rm We may also provide an explicit basis for $\mathfrak K^+_b$ in terms of the  Poincar\'e half-ball model
		\[
		\mathbb H^n_+=\{z=(z_1,\cdots,z_n)\in\mathbb R_n; |z|<1,z_n\geq 0\}
		\]
	with boundary $\partial\mathbb H^n_+=\{z\in\mathbb H^n_+;z_n=0\}$. In this representation, 
	\[
	b=\frac{4}{(1-|z|^2)^2}\delta,
	\]
	and the anti-de Sitter space ${\rm AdS}_+^{n,1}=\mathbb R\times \mathbb H^n_+$ is endowed with the metric
	\[
	\overline b=-\left(\frac{1+|z|^2}{1-|z|^2}\right)^2dz_{0}^2+b,\quad z_{0}\in\mathbb R.  
	\]
	It follows that 
	$
	\mathfrak K_b^+
	$
	is generated by 
	\[
	L_{0n}=\frac{1+|z|^2}{2}\partial_{z_n}-z_nz_j\partial_{z_j}, \quad L_{0n}|_{\partial\mathbb H^n_+}=\underbrace{\frac{1+|z|^2}{1-|z|^2}}_{=V_{(0)}}e_n,
	\]
	and
	\[
	L_{An}=z_A\partial_{z_n}-z_n\partial_{z_A},\quad L_{An}|_{\partial\mathbb H^n_+}=\underbrace{\frac{2z_A}{1-|z|^2}}_{=V_{(A)}}e_n.
	\]
	}
	\end{remark}

We define the energy-momentum vector $(\mathcal E, \mathcal P)\in \mathcal N^+_b\oplus\mathfrak K^+_b\cong\mathbb L^{1,n-1}\oplus \mathbb L^{1,n-1}$ by
\begin{equation}\label{pair:vect}
\begin{array}{cc}
\mathcal E_{a}=\mathfrak m_{(g,h, F)}(V_{(a)},0), & \mathcal P_{a}=\mathfrak m_{(g,h, F)}(0,W_{(a)}),
\end{array}
\end{equation}
where $V_{(a)}=x_a|_{\mathbb H^n_+}$ and $W_{(a)}=L_{an}|_{\mathbb H^n_+}$, $a=0,...,n-1$, are the generators of $\mathcal N^+_b$ and $\mathfrak K^+_b$, respectively.
\begin{corollary}
	If $(M,g,h,\Sigma)$ is an asymptotically hyperbolic initial data, then $\mathcal E$ and $\mathcal P$ are well defined up to composition with an element of ${\rm SO}_{1,n-1}$.
 In particular, the causal characters of $\mathcal E$ and $\mathcal P$ and the quantities
\begin{equation*}
\begin{array}{cc}
	\langle\langle \mathcal E,\mathcal E\rangle\rangle=\mathcal E_0^2-\mathcal E_1^2-...-\mathcal E_{n-1}^2,
	&
	\langle\langle \mathcal P,\mathcal P\rangle\rangle=\mathcal P_0^2-\mathcal P_1^2-...-\mathcal P_{n-1}^2
\end{array} 
\end{equation*}
	do not depend on the chart $F$ at infinity chosen to compute $(\mathcal E,\mathcal P)$.
\end{corollary}

The rest of this section is devoted to the proof of Proposition \ref{propo:mass:int}.
Define the constraint maps  $\Psi_\Lambda$ and $\Phi$  as in Section \ref{state}, namely
		\begin{equation*}
		\begin{array}{l}
		\disp \Psi_\Lambda:\mathcal M\times \mathcal S_2(M)\to C^{\infty}(M)\times\Gamma(T^*M), 
	\qquad \Phi 	: \mathcal{M} \times \mathcal{S}_2(M) \to C^{\infty}(\Sigma)\times \Gamma(T^*M|_{\Sigma}) \\[0.5cm]
		\Psi_\Lambda(g, h)
		=\left(
		\begin{array}{c}
		R_{ g}-2\Lambda-|h|^2_{ g}+({\rm tr}_{g}h)^2\\
		2\left(
		{\rm div}_{ g}h-d{\rm tr}_{ g}h
		\right)
		\end{array}
		\right), \qquad
		\Phi(g,h)=\left(
		\begin{array}{c}
		2H_g\\
		2 {{\varrho \righthalfcup (h-(\text{tr}_g h)g)}}
		\end{array}
		\right),
		\end{array}
		\end{equation*}
where $\mathcal S_2(M)$ is the space of symmetric bilinear forms on $M$, $\mathcal M\subset \mathcal S_2(M)$ is the cone of (positive definite) metrics and $\Lambda=0$ or $\Lambda=\Lambda_n$ according to the case.

We  follow the perturbative approach in \cite{Mi}. For notational simplicity we omit from now on the reference to the cosmological constant  in $\Psi_\Lambda$. 
Thus, in the asymptotic region we expand the constraint maps $(\Psi,\Phi)$ around $(\mathbb E^n_+,g_0, 0, \partial\mathbb E^n_+)$ to deduce that
\begin{equation}\label{eq:linear:appr}
\begin{array}{lcl}
\disp \Psi(g,h)& = & D\Psi_{(g_0,0)}(f,h)+\mathfrak R_{(g_0,0)}(f,h), \\[0.3cm]
\disp \Phi(g,h) & = & D\Phi_{(g_0,0)}(f,h) + \widetilde{\mathfrak R}_{(g_0,0)}(f,h),
\end{array}
\end{equation}
where $\mathfrak R_{(g_0,0)}$ and $\widetilde{\mathfrak R}_{(g_0,0)}$ are remainder terms that are at least quadratic in $(f,h)$ and 
\[
\begin{array}{cc}
\disp D\Psi_{(g_0,0)}(f,h)=\frac{d}{dt}\Big|_{t=0}\Psi(g_0+tf,th), &
\disp D\Phi_{(g_0,0)}(f,h)=\frac{d}{dt}\Big|_{t=0}\Phi(g_0+tf,th).
\end{array}
\]
Using formulas for metric variation in \cite{Mi, CEM}, we get
\begin{equation}\label{DpsiDphi}
\begin{array}{lcl}
\disp D\Psi_{(g_0,0)}(f,h)
& = & \disp 
\left(
\begin{array}{c}
{\rm div}_{g_0}({\rm div}_{g_0}f-d{\rm tr}_{g_0}f)-\langle\text{Ric}_{g_0},f\rangle_{g_0}\\
2({\rm div}_{g_0}h-d{\rm tr}_{g_0}h)
\end{array}
\right) \\[0.5cm]
D\Phi_{(g_0,0)}(f, h) & = & \disp \left(
\begin{array}{c}
\left( {\rm div}_{g_0}f - d{\rm tr}_{g_0}f\right)(\varrho_{g_0}) + {\rm div}_{\gamma_0} ((\varrho_{g_0} \righthalfcup f)^\downvdash)-\langle\Pi_{g_0},f\rangle_{\gamma_0}\\
2\varrho_{g_0} \righthalfcup (h-(\text{tr}_{g_0}h)g_0)
\end{array}
\right).
\end{array}
\end{equation}
As in \cite{Mi}, we obtain
\begin{equation}\label{inteparts}
\begin{array}{lcl}
\disp \langle D\Psi_{(g_0,0)}(f,h), (V,W) \rangle & = & \disp {\rm div}_{g_0} \big(\mathbb U_{(f,h)}(V,W)\big) + \langle (f,h), \mathscr{F}(V,W) \rangle_{g_0}, \\[0.3cm]
\end{array}
\end{equation}
where $\mathbb U_{(f,h)} (V,W)$ is defined by (\ref{defU}) and
\begin{equation}\label{Psistar}
\begin{array}{l}
\disp \mathscr{F}(V,W)=\left(
\begin{array}{c}
\nabla^2_{g_0}V-(\Delta_{g_0}V)g_0-V\text{Ric}_{g_0}\\
-\mathcal L_{W}g_0+2({\rm div}_{g_0}W)g_0
\end{array}
\right),
\end{array}
\end{equation}
is the formal adjoint of $D\Psi_{(g_0,0)}$. 

For large $r$ we set 
$S^{n-1}_{r,+}=\{x\in \mathbb E_+^n;|x|=r\}$,
where $|x|=\sqrt{x_1^2+...+x_n^2}$, and for $r'>r$ we define
\begin{equation*}
\begin{array}{cc}
A_{r,r'}=\{x\in \mathbb E^n_+;r\leq |x| \leq r'\}\,, & \Sigma_{r,r'}=\{x\in \partial \mathbb E_+^n;r\leq |x| \leq r'\}
\end{array}
\end{equation*}
so that 
$$\d A_{r,r'}=S^{n-1}_{r,+}\cup \Sigma_{r,r'}\cup S^{n-1}_{r',+}.$$ 
We represent by $\mu$ the outward unit normal vector field to $S^{n-1}_{r,+}$ or $S^{n-1}_{r',+}$, computed with respect to the reference metric $g_0$. Also, we consider $S^{n-2}_r=\partial S_{r,+}^{n-1} \subset  \partial \mathbb E^n_{+,r}$; see Figure \ref{fig1}. 
Using \eqref{eq:linear:appr}, \eqref{inteparts} and the divergence theorem, we have that 
\[
\mathcal A_{r,r'}(g,h):=\int_{A_{r,r'}} \Psi(g,h)(V,W) + \int_{\Sigma_{r,r'}} \Phi(g,h)(V,W)
\]
is given by 
\begin{eqnarray*}
\mathcal A_{r,r'}(g,h) & = & 
 \int_{S^{n-1}_{r',+}} \mathbb U_{(f,h)}(V,W)(\mu) - \int_{S^{n-1}_{r,+}} \mathbb U_{(f,h)}(V,W)(\mu)\nonumber \\
& &  + \int_{\Sigma_{r,r'}} \Big[ \langle D\Phi_{(g_0,0)}(f,h), (V,W) \rangle - \mathbb U_{(f,h)}(V,W)(\varrho_{g_0})\Big] \\
& &  + \int_{A_{r,r'}} \langle (f,h), \mathscr{F}(V,W) \rangle + \int_{A_{r,r'}} \mathfrak R_{(g_0,0)}(f,h)\\
& &  + \int_{\Sigma_{r,r'}} \widetilde{\mathfrak R}_{(g_0,0)}(f,h).
\end{eqnarray*}
From \eqref{DpsiDphi} and \eqref{defU}, the third integrand in the right-hand side above can be written as  
\[
f(\varrho_{g_0}, \nabla_{g_0} V - \nabla_{\gamma_0} V)  - (\tr_{g_0} f) d V(\varrho_{g_0}) -\langle V\Pi_{g_0}, f\rangle_{\gamma_0}+ {\rm div}_{\gamma_0}(V(\varrho_{g_0} \righthalfcup f)^\downvdash),
\]
which clearly equals 
\[
\langle - d V(\varrho_{g_0}) {\gamma_0} -V\Pi_{g_0}, f\rangle_{\gamma_0}+ {\rm div}_{\gamma_0}(V(\varrho_{g_0} \righthalfcup f)^\downvdash).
\]
Plugging this into the expression for $\mathcal A_{r,r'}$ above and using the divergence theorem, we eventually get
\begin{eqnarray*}
\mathcal A_{r,r'}(g,h) & = &	
 \int_{S^{n-1}_{r',+}} \mathbb U_{(f,h)}(V,W)(\mu) - \int_{S^{n-1}_{r,+}} \mathbb U_{(f,h)}(V,W)(\mu) \\
& &  + \int_{S^{n-2}_{r'}} V f(\varrho_{g_0}, \vartheta) - \int_{S^{n-2}_{r}} V f(\varrho_{g_0}, \vartheta)\\
& &  + \int_{A_{r,r'}} \mathfrak R_{(g_0,0)}(f,h) + \int_{\Sigma_{r,r'}} \widetilde{\mathfrak R}_{(g_0,0)}(f,h)\\
& &  + \int_{A_{r,r'}} \langle (f,h), \mathscr{F}(V,W) \rangle_{g_0} + \int_{\Sigma_{r,r'}} \langle - d V(\varrho_{g_0}) {\gamma_0} -V\Pi_{g_0}, f\rangle_{\gamma_0}.
\end{eqnarray*}
The last line vanishes due to (\ref{eq:static:gen}) and the fact that $W$ is Killing. Making use of the decay assumptions coming from Definitions \ref{def:as:hyp} and \ref{def:as:hyp2}, we are led to 
\begin{eqnarray}\label{ar_Arr'}
o_{r,r'}(1) & = &	
 \int_{S^{n-1}_{r',+}} \mathbb U_{(f,h)}(V,W)(\mu) - \int_{S^{n-1}_{r,+}} \mathbb U_{(f,h)}(V,W)(\mu) \\
& &  + \int_{S^{n-2}_{r'}} V f(\varrho_{g_0}, \vartheta) - \int_{S^{n-2}_{r}} V f(\varrho_{g_0}, \vartheta),\notag
\end{eqnarray}
where $o_{r,r'}(1)\to 0$ as $r,r' \to +\infty$. The first statement of Proposition \ref{propo:mass:int} follows at once. 

\begin{remark}
It is important to stress that, for \eqref{ar_Arr'} to hold, no boundary condition is imposed on the Killing field $W$. Indeed, as we shall see later, the requirement that $W\in\mathfrak K^+_{g_0}$ appearing in Definition \ref{def:charge} arises as a consequence of the spinor approach, and in particular it is not a necessary condition for the limit defining the energy-momentum vector to exist. 
\end{remark}

For the proof of (\ref{propo:mass:int:1}), we first consider the asymptotically hyperbolic case which is more involved than the asymptotically flat one. The next two results are Lemmas 3.3 and 3.2 in \cite{AdL}. 
\begin{lemma}\label{lem:rigidity}
If $\phi:\mathbb H^n_+\to\mathbb H^n_+$ is a diffeomorphism such that 
$$
\phi^*b=b+O(|y|^{-\kappa}), \:\:\text{as}\:|y|\to\infty,
$$ 
for some $\kappa>0$, then there exists an isometry A of $(\mathbb H^n_+,b)$ which preserves $\partial\mathbb H^n_+$ and satisfies 
$$
\phi=A+O(|y|^{-\kappa}),
$$
with similar estimate holding for the first order derivatives.
\end{lemma}
\begin{lemma}\label{lem:cancellation}
If $(V,W)\in \mathcal N_b^+\oplus \mathfrak K^+_b$ and $\zeta$ is a vector field on $\mathbb H^n_+$, tangent to $\partial\mathbb H^n_+$, then 
$$
\mathbb U_{(L_{\zeta}b,0)}(V,W)={\rm{div}}_b\mathbb V,
$$
with $\mathbb V_{ik}=V(\zeta_{i;k}-\zeta_{k;i})+2(\zeta_kV_i-\zeta_iV_k)$. 
\end{lemma}

We now follow the lines of \cite[Theorem 3.4]{AdL}. Suppose $F_1$ and $F_2$ are asymptotic coordinates for $(M,g,h,\Sigma)$ as in Definition \ref{def:as:hyp2} and set $\phi=F_1^{-1}\circ F_2$. It follows from Lemma \ref{lem:rigidity} that 
$\phi=A + O(|y|^{-\kappa})$, for some isometry $A$ of $(\mathbb H^n_+, b)$. By composing with $A^{-1}$, one can assume that $A$ is the identity map of $\mathbb H^n_+$. In particular, $\phi=\text{exp}\circ \zeta$ for some vector field $\zeta$ tangent to $\partial\mathbb H^n_+$. Set
\begin{equation*}
\begin{array}{cccc}
f_1=F_1^*g-b, & h_1=F_1^*h, & f_2=F_2^*g-b, & h_2=F_2^*h.
\end{array}
\end{equation*}
Then
$$
f_2-f_1=\phi^*F_1^*g-F_1^*g=\phi^*(b+f_1)-(b+f_1)=L_{\zeta}b+O(|y|^{-2\kappa}).
$$
Similarly, $h_2-h_1=O(|y|^{-2\kappa-1})$. This implies
$$
\mathbb U_{(f_2,h_2)}(V,W)-\mathbb U_{(f_1,h_1)}(V,W)
=\mathbb U_{(L_{\zeta}b,0)}(V,W)+O(|y|^{-2\kappa+1}).
$$
By Lemma \ref{lem:cancellation} and Stokes' theorem,
\begin{align}\label{eq:LXb}
\lim_{r\to\infty}\int_{S^{n-1}_{r,+}}\mathbb U_{(f_2,h_2)}(V,W)(\mu)
&-\lim_{r\to\infty}\int_{S^{n-1}_{r,+}}\mathbb U_{(f_1,h_1)}(V,W)(\mu)
\\
&=\lim_{r\to\infty}\int_{S^{n-1}_{r,+}}\mathbb U_{(L_{\zeta}b,0)}(V,W)(\mu)\notag
\\
&=\lim_{r\to\infty}\int_{\Sigma_r}\mathbb U_{(L_{\zeta}b,0)}(V,W)(\varrho_b)\notag
\end{align}
where $\Sigma_r=\{y\in \Sigma\:;\:|y|\leq r\}$. Observe that $\varrho_b \righthalfcup \mathbb V$ is tangent to the boundary and set $\beta=b|_{\partial\mathbb H^n_+}$. Direct computations yield
$$
\mathbb U_{(L_{\zeta}b,0)}(V,W)(\varrho_b)=\text{div}_{\beta}(\varrho_b \righthalfcup \mathbb V),
$$ 
so another integration by parts shows that the right-hand side of (\ref{eq:LXb}) is
$$
\lim_{r\to\infty}\int_{S^{n-2}_r}\mathbb V(\varrho_b,\vartheta)
=\lim_{r\to\infty}\int_{S^{n-2}_r}V(-\zeta_{\alpha;n})\vartheta_{\alpha}
=\lim_{r\to\infty}\int_{S^{n-2}_r}V(f_1-f_2)(\varrho_b,\vartheta).
$$
This proves (\ref{propo:mass:int:1}) for the asymptotically hyperbolic case. The asymptotically flat one is simpler: Lemma \ref{lem:cancellation} is not necessary because $\mathcal N^+_{\delta}\cong \mathbb R$, and  Lemma \ref{lem:rigidity} has a similar version in \cite[Proposition 3.9]{ABdL}.


\section{Spinors and the Dirac-Witten operator}\label{spinors}

Here we describe the so-called Dirac-Witten operator, which will play a central role in the proofs of Theorems \ref{main} and \ref{maintheohyp}. 
Although most of the material presented in this section is already available in the existing literature (see for instance \cite{CHZ,D,XD} and the references therein) we insist on providing a somewhat detailed account
as this will help us
to carefully keep track of the boundary terms, which are key for this paper.

We start with a few preliminary algebraic results. 
As before, let  $(\mathbb L^{n,1},\overline\delta)$ be the Minkowski space endowed with the standard orthonormal frame $\{\partial_{x_\alpha}\}_{\alpha=0}^n$. 
Let ${\rm Cl}_{n,1}$ be the Clifford algebra of the pair  $(\mathbb L^{n,1},\overline\delta)$ and ${\mathbb C}{\rm l}_{n,1}={\rm Cl}_{n,1}\otimes\mathbb C$ its complexification.  Thus, ${\mathbb C}{\rm l}_{n,1}$ is the unital complex algebra generated over $\mathbb L^{n,1}$ under the {\em Clifford relations}:
\begin{equation}\label{cliff}
X\cdot X'\cdot+X'\cdot X\cdot=-2\langle X,X'\rangle_{\overline\delta}, \quad X,X'\in\mathbb L^{n,1}, 
\end{equation}
where the 
dot represents Clifford multiplication.
Since ${\mathbb C}{\rm l}_{n,1}$ can be explicitly described in terms of matrix algebras, its representation theory is quite easy to understand. In fact, if $n+1$ is even then  ${\mathbb C}{\rm l}_{n,1}$ carries a unique irreducible representation whereas if $n+1$ is odd then it carries precisely two inequivalent irreducible representations.  

Let ${\rm SO}^0_{n,1}$ be the identity component of the subgroup of isometries of $(\mathbb L^{n,1},\overline\delta)$ fixing the origin. Passing to its simply connected double cover we obtain a Lie group homomorphism
\[
\overline \chi:{\rm Spin}_{n,1}\to {\rm SO}^0_{n,1}
\]
The choice of the time-like unit vector $\partial_{x_0}$ gives the identification 
\[
\mathbb R^n=\{X\in\mathbb L^{n,1};\langle X,\partial_{x_0}\rangle=0\}
\] 
so we obtain a  Lie group homomorphism
\[
\chi:=\overline\chi|_{{\rm Spin}_n}:{\rm Spin}_{n}\to {\rm SO}_{n}.
\]  
Hence, ${\rm Spin}_n$ is the universal double cover of ${\rm SO}_n$, the rotation group in dimension $n$.
Summarizing, we have the diagram
\[
\begin{array}{ccc}
{\rm Spin}_{n} & \stackrel{\overline\gamma}{\longrightarrow} & {\rm Spin}_{n,1}\\
\downarrow\chi & & \downarrow\overline\chi\\
{\rm SO}_n & \stackrel{\gamma}\longrightarrow & {\rm SO}_{n,1}^0
\end{array}
\]
where the horizontal arrows are inclusions and the vertical arrows are two-fold covering maps. 

We now recall that ${\rm Spin}_{n,1}$ can be realized as a multiplicative subgroup of ${\rm Cl}_{n,1}\subset{\mathbb C}{\rm l}_{n,1}$. Thus, by restricting any of the irreducible representations of ${\mathbb C}{\rm l}_{n,1}$ described above we obtain the so-called {\rm spin representation}
$\overline\sigma:{\rm Spin}_{n,1}\times S\to S$.
It turns out that the vector space $S$ comes with a natural positive definite, hermitian inner product $\langle\,,\rangle$ satisfying 
\[
\langle X\cdot\psi,\phi\rangle=\langle\psi,\theta(X)\cdot\phi\rangle, 
\]
where 
\[
\theta(a_0\partial_{x_0}+a_i\partial_{x_i})=a_0\partial_{x_0}-a_i\partial_{x_i}. 
\]
In particular, $\langle\,,\rangle$ is ${\rm Spin}_n$ but not ${\rm Spin}_{n,1}$-invariant; see \cite{D}. 
A way to partly remedy this is to consider another hermitean inner product on $S$ given by
\[
(\psi,\phi)=\langle\partial_{x_0}\cdot\psi,\phi\rangle,
\] 
which is clearly ${\rm Spin}_{n,1}$-invariant. But notice that $(\,,)$ is {\em not} positive definite. 
We remark that  $\partial_{x_0}\cdot$ is hermitean with respect to  $\langle\,,\rangle$ whereas $\partial_{x_i}\cdot$  is  skew-hermitean with respect to  $\langle\,,\rangle$. On the other hand, any $\partial_{x_\alpha}\cdot$ is hermitean with respect to $(\,,)$.

We now work towards globalizing the algebraic picture above. Consider a space-like embedding 
\[
i:(M^n,g)\hookrightarrow ({\overline M}^{n+1},\overline g)
\]
endowed with a time-like unit normal vector $e_0$. Here, $(\overline M,\overline g)$ is a Lorentzian manifold. Let $P_{\rm SO}(T\overline M)$ (respectively, $P_{\rm SO}(TM)$) be the principal ${\rm SO}^0_{n,1}-$ (respectively, ${\rm SO}_n-$) frame bundle of $TN$ (respectively, $TM$). Also, set
\[
\widetilde{P_{\rm SO}(T\overline M)}:=i^*P_{\rm SO}(T\overline M)
\]
to be the restricted principal ${\rm SO}_{n,1}^0$-frame bundle. In order to lift $\widetilde{P_{\rm SO}(TN)}$ to a principal ${\rm Spin}_{n,1}$-bundle we note that the choice of $e_0$ provides the identification
\[
\widetilde{P_{\rm SO}(T\overline M)}=P_{\rm SO}(TM)\times_{\gamma}{\rm SO}_{n,1}^0. 
\]
Now, as $M$ is supposed to be spin, there exists a twofold lift 
\[
P_{\rm Spin}(TM)\longrightarrow P_{\rm SO}(TM),
\] 
so $P_{\rm Spin}(TM)$ is the principal spin bundle of $TM$. We then set 
\[
\widetilde{P_{\rm Spin}(T\overline M)}:=P_{\rm Spin}(TM)\times_{\overline \gamma} {\rm Spin}_{n,1}, 
\]
which happens to be the desired lift of $\widetilde{P_{\rm SO}(T\overline M)}$. The corresponding {\em restricted} spin bundle is defined by means of the standard associated bundle construction, namely, 
\[
\mathbb S_{M}:=\widetilde{P_{\rm Spin}(T\overline M)}\times_{\overline \sigma} S.
\]
This comes endowed with the hermitean metric $(\,,)$ and a compatible connection $\overline \nabla$ (which is induced by the  extrinsic Levi-Civita connection $\overline \nabla$ of $(\overline M,\overline g)$).  Finally, we also have 
\[
\mathbb S_{M}=P_{\rm Spin}(TM)\times_{\sigma}S,
\] 
where $\sigma=\overline\sigma|_{{\rm Spin}_n}$. Hence, $\mathbb S_M$ is also endowed with the metric $\langle\,,\rangle$ and a compatible connection $\nabla$ (which is induced by the intrinsic Levi-Civita connection $\nabla$ of $(M,g)$) satisfying
\begin{equation}\label{eq_nablaintrinsic}
{\nabla_{e_i} = e_i + \frac{1}{4} \Gamma_{ij}^k e_j\cdot e_k\cdot} 
\end{equation}
In particular, $\overline\nabla$ is compatible with $(\,,)$ but {\em not} with $\langle\,,\rangle$.
We finally remark that in terms of an adapted frame  $\{e_\alpha\}$ there holds
\begin{equation}\label{spingauss}
\overline\nabla_{e_i}=\nabla_{e_i}-\frac{1}{2}h_{ij}e_0\cdot e_j\cdot,
\end{equation}
where $h_{ij}=g(\overline\nabla_{i}e_0,e_j)$ are the components of the second fundamental form. This is the so-called {\em spinorial Gauss formula}.

We are now ready to introduce the main character of our story. 

\begin{definition}
	The Dirac-Witten operator $\overline{\mathcal D}$ is defined by the composition
	\[
	\Gamma(\mathbb S_M)\stackrel{\overline\nabla}{\longrightarrow}\Gamma(TM\otimes \mathbb S_M)\stackrel{\cdot}{\longrightarrow}\Gamma(\mathbb S_M)\,.
	\]
	Locally, 
	\[
	\overline{\mathcal D}=e_i\cdot\overline\nabla_{e_i}.
	\]
\end{definition}

The key point here is that $\overline{\mathcal D}$ has the same symbol as the intrinsic Dirac operator ${\mathcal D} =e_i\cdot\nabla_{e_i}$ 
but in its definition $\overline\nabla$ is used instead of $\nabla$. 

Usually we view $\overline{\mathcal D}$ as acting on spinors satisfying a suitable boundary condition along $\Sigma$.

In what follows we discuss the one to be used for the asymptotically flat case. The one for the asymptotically hyperbolic will be discussed in Section \ref{secthyp}.

\begin{definition}\label{defmit}
	Let $\omega=i\varrho\,\cdot$ be the (pointwise) hermitean involution acting on $\Gamma({\mathbb S_M}|_{\Sigma})$. We say that a spinor $\psi\in \Gamma(\mathbb S_M)$ satisfies the \it{MIT bag}  boundary condition if any of the identities 
	\begin{equation}\label{mitcond}
	\omega\psi=\pm\psi
	\end{equation}
	holds
	along $\Sigma$. 
	\end{definition}

\begin{remark}\label{rem_MIT}
\emph{
	Note that a spinor $\psi$ satisfying a MIT bag boundary condition also enjoys $( \varrho \cdot \psi, \psi) = 0$ on $\Sigma$. Indeed, the fact that $\varrho\,\cdot$ is Hermitian for $( \, , \,)$ implies
\begin{equation*}
(\varrho \cdot \psi, \psi) = \big(\varrho \cdot (\pm i \varrho\cdot\psi), \pm i \varrho\cdot\psi \big) = - ( \psi, \varrho \cdot \psi) = - (\varrho \cdot \psi, \psi).
\end{equation*}	
}
\end{remark}

\begin{proposition}\label{sefaddir}
	Let $\overline{\mathcal D}_\pm$ be the Dirac-Witten operator acting on spinors satisfying the MIT bag boundary condition (\ref{mitcond}). Then $\overline{\mathcal D}_+$ and $\overline{\mathcal D}_-$  are adjoints to each other with respect to $\langle\,,\rangle$.
	\end{proposition}

\begin{proof}
	 We use a local frame such that $\nabla_{e_i}e_j|_p=0$ for a given $p\in M$. It is easy to check that at this point we also have $\overline\nabla_{e_i}e_j= {h_{ij}e_0}$ and $\overline\nabla_{e_i}e_0= {h_{ij}e_j}$. 
	Now, given spinors $\phi,\xi\in\Gamma(\mathbb S_M)$, consider the $(n-1)$-form
	\[
	\widehat \theta=\langle e_i\cdot\phi,\xi\rangle e_i\righthalfcup dM.  
	\]
	Thus,
	\[
	d\widehat\theta=e_i\langle e_i\cdot\phi,\xi\rangle dM=e_i(e_0\cdot e_i\cdot\phi,\xi) dM=
	e_i(e_i\cdot\phi,e_0\cdot\xi)dM. 
	\]
	Using that $\overline\nabla$ is compatible with $(\,,)$, we have
	\begin{eqnarray*}
		d\widehat\theta & = & \big((\overline\nabla_{e_i}e_i\cdot\phi,e_0\cdot\xi) + (e_i\cdot\overline\nabla_{e_i}\phi,e_0\cdot\xi)\\
		& & \quad +(e_i\cdot\phi,\overline\nabla_{e_i}e_0\cdot\xi)+(e_i\cdot\phi,e_0\cdot\overline\nabla_{e_i}\xi) \big) dM\\
		& = & \big( h_{ii}(e_0\cdot\phi,e_0\cdot\xi)+(\overline{\mathcal D}\phi,e_0\cdot\xi)\\
		&  & \quad + h_{ij}(e_i\cdot\phi,e_j\cdot\xi)-(e_0\cdot\phi,e_i\cdot\overline\nabla_{e_i}\xi)\big)dM\\
		& = & (h_{ii}(\phi,\xi)+ (e_0\cdot\overline{\mathcal D}\phi,\xi)\\
		& & \quad + h_{ij}(e_j\cdot e_i\cdot\phi,\xi)-(e_0\cdot\phi,\overline{\mathcal D}\xi))dM.
	\end{eqnarray*}
	Now, the first and third terms cancel out due to the Clifford relations (\ref{cliff}) so we end up with
	\begin{equation}\label{selfpoint}
	d\widehat\theta=(\langle\overline{\mathcal D}\phi,\xi\rangle-\langle \phi,\overline{\mathcal D}\xi\rangle)dM. 
	\end{equation}
	Hence, assuming that $\phi$ and $\psi$ are compactly supported we get
	\begin{eqnarray*}
	\int_M\langle\overline{\mathcal D}\phi,\xi\rangle dM-\int_M \langle \phi,\overline{\mathcal D}\xi\rangle dM & = & \int_\Sigma \langle e_i\cdot\phi,\xi\rangle e_i\righthalfcup dM\\
	& = & \int_\Sigma \langle \varrho\cdot\phi,\xi\rangle d\Sigma.
	\end{eqnarray*}
where we used an adapted frame such that $e_n=\varrho$. 
If $\omega\phi=\phi$ and $\omega\xi=-\xi$ we have
\begin{eqnarray*}
	\langle  \varrho\cdot\phi,\xi\rangle & = & \langle \varrho\cdot(i \varrho\cdot\phi),- i \varrho\cdot\xi\rangle\\
	& = & \langle \phi,\varrho\cdot \xi\rangle\\
	& = & -\langle \varrho\cdot \phi,\xi\rangle,
\end{eqnarray*}
that is, $\langle  \varrho\cdot\phi,\xi\rangle=0$.
\end{proof}

\begin{proposition}\label{moredir}
	Given a spinor $\psi \in \Gamma(\mathbb{S}_M)$, define the $(n-1)$-forms
	\[
	\theta=\langle e_i\cdot\overline{\mathcal D}\psi,\psi\rangle e_i\righthalfcup dM, \quad \eta=\langle\overline\nabla_{e_i}\psi,\psi\rangle  e_i\righthalfcup dM, 
	\]
	then
	\begin{equation}\label{1}
	d\theta=\left(\langle\overline{\mathcal D}^2\psi,\psi\rangle-|\overline{\mathcal D}\psi|^2\right)dM,
	\end{equation}
	and
	\begin{equation}\label{2}
	d\eta=\left(-\langle \overline{\nabla}^*\overline{\nabla}\psi,\psi\rangle+|\overline{\nabla}\psi|^2\right)dM, 
	\end{equation}
	where $\overline\nabla^*\overline\nabla=\overline{\nabla}^*_{e_i}\overline\nabla_{e_i}$ is the Bochner Laplacian acting on spinors. Here,
	\[
	\overline\nabla_{e_i}^*=-\overline\nabla_{e_i} {+} h_{ij}e_j\cdot e_0\cdot
	\]
	is the formal adjoint of $\overline{\nabla}_{e_i}$ with respect to $\langle\,,\rangle$.
\end{proposition}

\begin{proof}		
	 By setting $\phi= \overline{\mathcal D}\psi$ and $\xi=\psi$ in (\ref{selfpoint}), (\ref{1}) follows. To prove (\ref{2}) we note that 
	{
	\[
	d\eta=e_i\langle \overline\nabla_{e_i}\psi,\psi\rangle dM=e_i(e_0\cdot\overline\nabla_{e_i}\psi,\psi)dM=e_i(\overline\nabla_{e_i}\psi,e_0\cdot\psi)dM,
	\]  
	so that 
	\begin{eqnarray*}
		d\eta & = & ((\overline{\nabla}_{e_i}\overline\nabla_{e_i}\psi,e_0\cdot\psi)+(\overline\nabla_{e_i}\psi,\overline\nabla_{e_i}e_0\cdot\psi)+(\overline\nabla_{e_i}\psi,e_0\cdot\overline\nabla_{e_i}\psi))dM\\ 
		& = & ((e_0\cdot\overline\nabla_{e_i}\overline{\nabla}_{e_i}\psi,\psi) + h_{ij}(\overline\nabla_{e_i}\psi,e_j\cdot \psi)+(e_0\cdot\overline\nabla_{e_i}\psi,\overline\nabla_{e_i}\psi))dM.
	\end{eqnarray*}
	The term in the middle equals
	\[
	h_{ij}(e_j\cdot\overline\nabla_{e_i}\psi,\psi)= h_{ij}\langle e_0\cdot e_j\cdot\overline\nabla_{e_i}\psi,\psi\rangle= - h_{ij}\langle e_j\cdot e_0\cdot\overline\nabla_{e_i}\psi,\psi\rangle,
	\]
	so we end up with
	\[
	d\eta  =  (\langle-(\underbrace{-\overline\nabla_{e_i} + h_{ij}e_j\cdot e_0\cdot)}_{\overline\nabla^*_{e_i}}\overline\nabla_{e_i}\psi,\psi\rangle+|\overline\nabla\psi|^2)d{\rm vol},
	\]
}
	which completes the proof of (\ref{2}).
\end{proof}

Another key ingredient is the following Weitzenb\"ock-Lichnerowicz formula.

\begin{proposition}
	One has 
	\begin{equation}\label{wetz}
	\overline{\mathcal D}^2=\overline\nabla^*\overline\nabla+\mathcal R,
	\end{equation}
	where the symmetric endomorphism $\mathcal R$ is given by
	\[
	\mathcal R=\frac{1}{4}(R_{\overline g}+2{{\rm Ric}_{\overline g}}_{0\alpha}e_\alpha\cdot e_0 \cdot).
	\]
\end{proposition}

\begin{proof}
	See \cite{PT}.
\end{proof}

\begin{remark}\label{claim}{\rm 
	We recall that the DEC (\ref{intdec}) with $\Lambda=0$ implies that $\mathcal R\geq 0$. Indeed, a simple computation shows that 
	\[
	\mathcal R=\frac{1}{2}(\rho_0+J_{i}e_i\cdot e_0\cdot).
	\]
	Hence,
	if $\phi$ is an eigenvector of $\mathcal R$ then it is an eigenvector of $\mathcal J:=J_{i}e_i\cdot e_0\cdot$ as well, say with eigenvalue equal to $\lambda$. Thus,
	\begin{eqnarray*}
		\lambda^2|\phi|^2 & = & \langle J_{i}e_i\cdot e_0\cdot\phi , J_{j}e_j\cdot e_0\cdot\phi\rangle\\
		& = & J_{i}J_{j}\langle e_i\cdot\phi,e_j\cdot\phi\rangle\\
		& = & -J_{i}J_{j}\langle e_j\cdot e_i\cdot\phi,\phi\rangle\\
		& = & \left(\sum_i J_{i}^2\right)|\phi|^2, 
	\end{eqnarray*}
	so that $\lambda=\pm|J|$. The claim follows.}
	\end{remark}

By putting together the results above, and using a standard approximation procedure, we obtain a fundamental integration by parts formula which will play a key role in the proof of Theorem \ref{main}. Hereafter, $H^k(\mathbb{S}_M)$ denotes the Sobolev space of spinors whose derivatives up to order $k$ are in $L^2$. We refer to \cite{ABdL,GN} for its definition and basic properties. 

\begin{proposition}\label{put}
	If $\psi\in H^2_\loc(\mathbb S_M)$ and  $\Omega\subset M$ is compact then 
	\begin{equation}\label{intpart1}
	\int_{\Omega}(|\overline\nabla\psi|^2+\langle \mathcal R\psi,\psi\rangle-|\overline{\mathcal D}\psi|^2)dM=\int_{\partial\Omega}\langle(\overline\nabla_{e_i}+e_i\cdot\overline{\mathcal D}) \psi,\psi\rangle e_i\righthalfcup dM. 
	\end{equation}
	\end{proposition}

The next proposition describes how the operator in the right-hand side of (\ref{intpart1})  decomposes into its intrinsic and extrinsic components. This is a key step toward simplifying our  approach to Theorems \ref{main} and \ref{maintheohyp} as it allows us to make use of the ``intrinsic''  computations in \cite{ABdL, AdL}.

\begin{proposition}\label{decomp}
	One has 
	\begin{equation}\label{decom2}
	\overline\nabla_{e_i}+e_i\cdot\overline{\mathcal D}=  {\nabla_{e_i} +e_i\cdot{\mathcal D}}-{\frac{1}{2}
		\pi_{ij}e_0\cdot e_j\cdot},
	\end{equation}
	where $\mathcal D=e_j\cdot\nabla_{e_j}$ is the intrinsic Dirac operator.
	\end{proposition}

\begin{proof}
	From (\ref{spingauss}) we have 
\begin{eqnarray*}
	\overline{\mathcal D} \: =\:  e_k\cdot 	\overline\nabla_{e_k}
	& = & e_k\cdot\left(\nabla_{e_k}-\frac{1}{2}h_{kl}e_0\cdot e_l\cdot\right)\\
	& = & {\mathcal D}-\frac{1}{2}h_{kl}e_k\cdot e_0\cdot e_l\cdot
\end{eqnarray*}
so that 
\begin{eqnarray*}
		\overline\nabla_{e_i}+e_i\cdot	\overline{\mathcal D} & = & \nabla_{e_i}-\frac{1}{2}h_{ij}e_0\cdot e_j\cdot+e_i\cdot\left({\mathcal D}-\frac{1}{2}h_{kl}e_k\cdot e_0\cdot e_l\cdot\right)\\
	& = &  \nabla_{e_i}-\frac{1}{2}h_{ij}e_0\cdot e_j\cdot+e_i\cdot{\mathcal D}-\frac{1}{2}h_{kl}e_i\cdot e_k\cdot e_0\cdot e_l\cdot \\
	& = &  \nabla_{e_i} +e_i\cdot{\mathcal D}-\frac{1}{2}h_{ij}e_0\cdot e_j\cdot+ \frac{1}{2}h_{kl}e_i\cdot e_k\cdot e_l\cdot e_0\cdot\\
	& = &  \nabla_{e_i} +e_i\cdot{\mathcal D}-\frac{1}{2}h_{ij}e_0\cdot e_j\cdot +
	\frac{1}{2}h_{kk}e_0\cdot e_i\cdot\\
	& & \quad +\frac{1}{2}\sum_{k\neq l}h_{kl}e_i\cdot(\underbrace{e_k\cdot e_l\cdot+e_l\cdot e_k\cdot}_{=0})e_0\cdot,
\end{eqnarray*}
and the result follows.
\end{proof}


\section{The asymptotically flat case}\label{proofmain}

In this section, we prove Theorem \ref{main}. Assume that the embedding $(M,g)\hookrightarrow (\overline M,\overline g)$ is {asymptotically flat} in the sense of Definition \ref{def:as:hyp}. In particular, in the asymptotic region we have 
$g_{ij}=\delta_{ij}+a_{ij}$ with
\[
|a_{ij}|+|x||\partial_{x_k} a_{ij}|+|x|^2|\partial_{x_k}\partial_{x_l}a_{ij}|=O(|x|^{-\tau}).
\]
where $\tau>(n-2)/2$. 
Thus, we may orthonormalize the standard frame $\{\partial_{x_i}\}$ by means of 
\begin{equation}\label{orthon}
e_i=\partial_{x_i}-\frac{1}{2}a_{ij}\partial_{x_j}+O(|x|^{-\tau})=\partial_{x_i}+O(|x|^{-\tau}),
\end{equation}
and we can further assume that $e_n = \varrho$ is the inward pointing normal vector to $\Sigma \hookrightarrow M$. 
We denote with a hat the extension, to the spinor bundle, of the linear isometry 
\begin{equation}\label{lineariso}
T\mathbb{R}^n_{+,r_0}\to TM_{\rm ext}, \qquad X^i \partial_{x_i} \mapsto X^i e_i.
\end{equation}
Note that a spinor $\phi$ on $\mathbb{S}_{\mathbb{R}^n_{+, r_0}}$ satisfies the MIT bag boundary condition (\ref{mitcond}) with $\omega = i \partial_{x_n} \cdot$, if and only if its image $\hat \phi$ on $\mathbb{S}_{M_{\rm{ext}}}$ satisfies (\ref{mitcond}) with $\omega=i e_n\cdot$.

We begin by specializing the identity in Proposition \ref{put} to the  case $\Omega=\Omega_r$, the compact region in an initial data set $(M,g,h,\Sigma)$ determined by the coordinate hemisphere $S_{r,+}^{n-1}$; see Figure \ref{fig1}. Notice that $\partial \Omega_r=S^{n-1}_{r,+}\cup \Sigma_r$, where $\Sigma_r$ is the portion of $\Sigma$ contained in $\Omega_r$.

\begin{proposition}\label{intpart11}
		Assume that $\psi\in H^2_\loc(\mathbb S_M)$ satisfies the boundary condition (\ref{mitcond}) along $\Sigma$. Then
		\begin{eqnarray}\label{intpartf0}
\int_{\Omega_r}(|\overline\nabla\psi|^2+\langle \mathcal R\psi,\psi\rangle-|\overline{\mathcal D}\psi|^2)dM & = & \int_{S^{n-1}_{r,+}}\langle(\overline\nabla_{e_i}+e_i\cdot \overline{\mathcal D}) \psi,\psi\rangle e_i\righthalfcup dM\nonumber\\ 
& & \quad -\frac{1}{2}\int_{\Sigma_r}\langle(H_g+\mathcal U)\psi,\psi\rangle d\Sigma,
\end{eqnarray}	
where $\mathcal U=\pi_{An}e_0\cdot e_A\cdot$ and $e_n=\varrho$. 
\end{proposition}

\begin{proof}
	We must work out the contribution of the right-hand side of (\ref{intpart1}) over $\Sigma_r$. By (\ref{decom2}),
	\begin{eqnarray*}
	\int_{\Sigma_r}\langle(\overline\nabla_{e_i}+e_i\cdot\overline{\mathcal D}) \psi,\psi\rangle e_i\righthalfcup dM  
	& = & 
	\int_{\Sigma_r}\langle(\nabla_{e_i}+e_i\cdot{\mathcal D}) \psi,\psi\rangle e_i\righthalfcup dM\\
	& &  -\frac{1}{2}\int_{\Sigma_r}\langle\mathcal U\psi,\psi\rangle d\Sigma
	     -\frac{1}{2}\int_{\Sigma_r}\pi_{nn}\langle e_0\cdot e_n\cdot\psi,\psi\rangle d\Sigma.
	\end{eqnarray*}
Because of Remark \ref{rem_MIT}, the MIT bag boundary condition guarantees that 
	\[
	\langle e_0\cdot e_n\cdot\psi,\psi\rangle = (\varrho \cdot \psi, \psi) = 0, 
	\]
and the last integral vanishes. On the other hand, it is known that 
\[
\int_{\Sigma_r} \langle (\nabla_{e_i}+e_i\cdot{\mathcal D}) \psi,\psi\rangle e_i\righthalfcup dM=\int_{\Sigma_r}\left\langle \left(D^{\intercal}-\frac{H_g}{2}\right)\psi,\psi\right\rangle d\Sigma,
\]
where $D^\intercal$ is a certain Dirac-type operator associated to the embedding $\Sigma\hookrightarrow M$; see \cite[p.697]{ABdL}  for a detailed discussion of this operator. However, $D^\intercal$ intertwines the projections defining the boundary conditions and this easily implies that 
$\langle D^\intercal\psi,\psi\rangle=0$. 
	\end{proof}

\begin{remark}\label{decbdr}{\rm 
	Let $\psi \in \Gamma(\mathbb{S}_{\Sigma})$ be an eigenvector of the linear map 
	$\mathcal U = \pi_{An} e_0\cdot e_A$, say with eigenvalue $\lambda$. We then have
	\begin{eqnarray*}
		\lambda^2|\psi|^2 
		& = & \langle \pi_{An}e_0\cdot e_A\cdot\psi,\pi_{Bn}e_0\cdot e_B\cdot\psi\rangle\\
		& = & 
		\pi_{An}\pi_{Bn}\langle e_A\cdot\psi,e_B\cdot\psi\rangle\\
		& = & 
		-\pi_{An}\pi_{Bn}\langle e_B\cdot e_A\cdot\psi,\psi\rangle\\
		& = & \left(\sum_A \pi_{An}^2\right)|\psi|^2.
	\end{eqnarray*}
	Thus, the eigenvalues  of $\mathcal U$ are 
	$\pm\sqrt{\sum_A \pi^2_{An}}=\pm|(\varrho\righthalfcup \pi)^\downvdash|$. In particular, if the DEC (\ref{bddectan2}) holds then $H_g+\mathcal U\geq 0$. }
	\end{remark}

The next step  involves a judicious choice of a spinor $\psi$ to be used in (\ref{intpartf0}) above. Denote with $\Gamma_c(\mathbb{S}_M)$ the space of smooth spinors with compact support in $M$, and with $\mathscr{H}$ its completion with respect to the norm 
	\[
	\|\psi\|_{\mathscr{H}} : = \|\nabla \psi\|_2 + \left\| \frac{\psi}{r} \right\|_{2},
	\]
where $r(x)>0$ for $x \in M$ and $r(x) = |x|$ in a fixed asymptotic chart, and $\|\cdot \|_2$ is the $L^2$-norm on the entire $M$.

\begin{proposition}\label{surjdir}
	Assume that the DECs (\ref{intdec2}) and (\ref{bddectan2}) hold. If $\eta \in L^2(\mathbb S_M)$, there exists a unique $\varphi\in \mathscr{H}$ solving any of the boundary value problems
	\begin{equation}\label{eq:phi}
	\left\{
	\begin{array}{rcll}
	\overline{\mathcal D}\varphi & = & \eta & \qquad \text{on } \, M\\
	\omega \varphi & = & \pm\varphi & \qquad \text{on } \, \Sigma.
	\end{array}
	\right.
	\end{equation}
Moreover, $\varphi \in H^2_\loc(\mathbb S_M)$ whenever $\eta \in H^1_\loc(\mathbb S_M)$.
\end{proposition}

\begin{proof}
We adapt arguments in \cite{GN,PT,bartnikchrusciel}. First, observe the validity of the Euclidean Hardy (weighted Poincar\'e) inequality 
	\begin{equation}\label{hardy}
	\int_{\R^n_+} \frac{(n-2)^2}{4|x|^2} |\psi|^2 \le \int_{\R^n_+} |\nabla \psi|^2 \qquad \forall \, \psi \in \Gamma_c(\mathbb S_{\R^n_+}).  
	\end{equation}
Indeed, it is sufficient to extend $\varphi$ by reflection across $\partial \R^n_+$ and to apply the Hardy inequality for spinors in $\mathbb S_{\R^n}$. We sketch a brief argument to prove the latter, for the sake of completeness: consider a constant spinor $\psi_0$ of norm $1$ on $\R^n_+$ and set $\tau : = w \psi_0$, with $w(x) = |x|^{\frac{2-n}{2}}$ being the square root of the Green kernel of $-\Delta_\delta$ on $\R^n$ with pole at the origin. A computation shows that 
	\[
	\mathcal{D} \tau = \nabla w \cdot \psi_0, \qquad \mathcal{D}^2 \tau = \frac{(n-2)^2}{4|x|^2}\tau \qquad \text{on } \, \R^n,
	\]
Given $\psi \in \Gamma_c(\mathbb S_{\R^n})$ supported outside of the origin, integrating by parts against $(|\psi|^2/w) \psi_0$ with respect to the Euclidean measure we obtain
	\[
	\begin{array}{lcl}
	\disp \int_{\R^n} \frac{(n-2)^2}{4|x|^2} |\psi|^2 & \le & \disp \int_{\R^n} \frac{|\psi|^2}{w} \langle \psi_0, \mathcal{D}^2 \tau \rangle \\[0.4cm]
	& = & \disp \int_{\R^n} \frac{|\psi|^2}{w^2} \langle \psi_0, \nabla w \cdot \mathcal{D}\tau \rangle - 2 \int_{\R^n} \frac{\langle \psi, \nabla_{\partial_{x_i}} \psi \rangle}{w} \langle \psi_0, \partial_{x_i} \cdot \mathcal{D}\tau \rangle \\[0.4cm]
	& \le & \disp - \int_{\R^n} \frac{|\psi|^2}{w^2} |\nabla w|^2 + 2 \int_{\R^n} \frac{|\psi| | \nabla \psi|}{w} |\nabla w| \\[0.4cm]
	& \le & \disp \int_{\R^n} |\nabla \psi|^2, 
	\end{array}
	\]
where in the second and third lines we used, respectively, Cauchy-Schwarz and Young inequalities. Pick an asymptotic chart $\R^{n}_{+,r_0} \to M_{{\rm ext}}$, and write for convenience $M_R = \{x \in M : |x|>R\}$. Let $\delta$ denote a metric on $M$ that is Euclidean on $M_{\rm ext}$, and let $r$ be a positive function on $M$ that equals $|x|$ on $M_{\rm ext}$. Define also a family of cut-offs $\{\eta_R\} \subset C^\infty_c(M)$ satisfying
	\begin{equation}\label{def_etaR}
0 \le \eta_R \le 1, \qquad \eta_R = 1 \ \text{ on } \, M \backslash M_R, \qquad \mathrm{spt} \, \eta_R \subset M \backslash M_{2R}, \qquad |\nabla \eta_R| \le \frac{\kappa}{R}	
	\end{equation}
for some constant $\kappa>0$. With the aid of the isometry \eqref{lineariso}, we deduce from \eqref{hardy} the existence of a constant $C>0$ only depending on $\|g-\delta\|_\infty$ such that
	\begin{equation}\label{wp_infty}
	\left\| \frac{\varphi}{r} \right\|_2 \le C \|\nabla \varphi\|_2 \qquad \forall \, \varphi \in \Gamma_c(\mathbb S_{M_R}),
	\end{equation}
where now the norms and connections are taken with respect to $g$. A gluing argument with the aid of the cutoff $\eta_R$ in \eqref{def_etaR} and of the Poincar\'e inequality on $M \backslash M_{R}$ guarantees that the weighted Poincar\'e inequality
	\begin{equation}\label{eq_wp_global}
	\left\| \frac{\varphi}{r} \right\|_{2} \le C \|\nabla \varphi\|_{2} \qquad \forall \, \varphi \in \Gamma_c(\mathbb S_M)
	\end{equation}
holds for some constant $C>0$. From the spinorial Gauss Equation and using $|h| \le C_1 r^{-\tau-1}$, 
	\begin{equation}\label{eq_nablabarnabla}
	\begin{array}{lcl}
	\|\overline \nabla \psi\|_2 & \le & \disp \|\nabla \psi\|_2 + \||h| \psi\|_2 \le \|\nabla \psi\|_2 + C \left\|\frac{\psi}{r}\right\|_2 \le C_1 \|\psi\|^2_{\mathscr{H}} \qquad \forall \, \psi \in \Gamma_c(\mathbb S_M), \\[0.3cm]	
	\|\overline \nabla \psi\|_2 & \ge & \disp \|\nabla \psi\|_2 - \||h| \psi\|_2 \ge \|\nabla \psi\|_2 - \frac{C_1}{R^\tau} \left\|\frac{\psi}{r}\right\|_2 \qquad \forall \, \psi \in \Gamma_c(\mathbb S_{M_R}).
	\end{array}
	\end{equation}
In particular, for $R$ large enough, because of \eqref{eq_wp_global} the term containing $R^{-\tau}$ in the second line can be absorbed into the norm $\|\cdot\|_{\mathscr{H}}$, yielding to
	\begin{equation}\label{low_MR}
	C_2 \|\psi\|_{\mathscr{H}}^2 \le \|\overline \nabla \psi\|_2 \qquad \forall \psi \in \Gamma_c(\mathbb S_{M_R}).
	\end{equation}
The gluing argument in \cite[Lemma 5.5]{PT} (cf. also \cite[Theorem 9.5]{bartnikchrusciel}) guarantees that
	\[
	C_3 \|\psi\|_{\mathscr{H}}^2 \le \|\overline \nabla \psi\|_2 \le C_1 \|\psi\|_{\mathscr{H}}^2 \qquad \forall \psi \in \Gamma_c(\mathbb S_{M})
	\]
for some constant $C_3>0$. Define $\mathscr{H}_\pm$ as the $\|\cdot \|_{\mathscr{H}}$-closure of the subspaces 
	\[
	\Gamma_c(\mathbb S_M)_\pm  : = \big\{ \psi \in \Gamma_c(\mathbb{S}_M) \ : \ \omega \psi = \pm \psi \ \text{ on } \, \Sigma \big\},
	\]
respectively.

From the very definition of $\overline{\mathcal{D}}$, $\|\overline{\mathcal D} \psi\|^2_2 \le n\|\overline{\nabla} \psi\|^2_2 \le C \|\psi\|^2_{\mathscr{H}}$, thus $\overline{\mathcal{D}}$ extends to a continuous operator $\overline{\mathcal{D}} : \mathscr{H} \to L^2(\mathbb S_M)$. The identity
	\begin{equation}\label{ide_spectral}
	\begin{array}{lcl}
\disp \int_M(|\overline\nabla\psi|^2+\langle \mathcal R\psi,\psi\rangle-|\overline{\mathcal D}\psi|^2)dM & = & \disp -\frac{1}{2}\int_{\Sigma}\langle(H_g+\mathcal U)\psi,\psi\rangle d\Sigma \\[0.4cm]
 & & \disp \qquad \forall \, \psi \in \Gamma_c(\mathbb S_M), \ \omega \psi = \pm \psi \ \text{on } \, \Sigma
	\end{array}
	\end{equation}	
in Proposition \ref{intpart11} together with the DECs imply that $\|\overline{\mathcal{D}}\psi \|_2 \ge \|\overline \nabla \psi\|_2$ for every $\psi \in \Gamma_c(\mathbb{S}_M)_\pm$, hence 
	\begin{equation}\label{eq_coerci}
	\disp \|\overline{\mathcal D} \psi\|_2 \ge C'\|\psi\|_{\mathscr{H}} \qquad \forall \, \psi \in \mathscr{H}_\pm.
	\end{equation}
To analyse more closely the boundary trace of $\psi \in \mathscr{H}$, define the conformally deformed metric $\hat g : = r^{-\frac{4}{n}} g = e^{2u} g$, where we set for convenience $u : = - \frac{2}{n}\log r$. Denote quantities referring to $\hat g$ with a hat and let $\hat{\mathbb S}_M$ be the space of spinors for $\hat g$.
To each $X \in TM$, set $\hat X = e^{-u} X$. It is known (cf. \cite[Section 4]{hijazi}) that there exists an isomorphism $\hat{}  : \mathbb S_M \to \hat{\mathbb S}_M$, which is an isometry on fibers and satisfies
	\[
	\widehat{X \cdot \psi} = \hat X \, \hat{\cdot} \, \hat \psi, \qquad \hat \nabla_X \hat \psi = \widehat{\nabla_X \psi} - \frac{1}{2} \widehat{ X \cdot \nabla u \cdot \psi} - \frac{1}{2} X(u) \hat \psi.
	\]
Let $\{e_i\}$ be an orthonormal frame for $g$, and $\hat e_i = e^{-u}e_i$ the corresponding orthonormal frame for $\hat g$. We compute
	\[
	\begin{array}{lcl}
	\disp |\hat \nabla \hat \psi|_{\hat g}^2 & = & \disp \sum_i |\hat \nabla_{\hat e_i} \hat \psi|_{\hat g}^2 \le 8 e^{-2u} \sum_i \left( |\widehat{\nabla_{e_i} \psi}|_{\hat g}^2 + \frac{1}{4} |\widehat{ e_i \cdot \nabla u \cdot \psi}|_{\hat g}^2 + \frac{1}{4} |e_i(u)|^2 |\hat \psi|_{\hat g}^2 \right) \\[0.4cm]
	& \le & \disp 8 e^{-2u} \Big( |\nabla \psi|^2 + |\nabla u|^2 |\psi|^2\Big).
	\end{array}
	\]
Therefore, 
	\begin{equation}\label{eq_boundbordo}
	\begin{array}{lcl}
	\disp \int_M |\hat \psi|^2_{\hat g} dM_{\hat g} & = & \disp \int_M |\psi|^2 r^{-2} dM \le \|\psi\|^2_{\mathscr{H}}, \\[0.3cm]
	\disp \int_M |\hat \nabla \hat \psi|_{\hat g}^2 dM_{\hat g} & \le & \disp C \int_M e^{-2u} \big( |\nabla \psi|^2 + |\nabla u|^2 |\psi|^2\Big)e^{nu} dM \le C \|\psi\|^2_{\mathscr{H}},
	\end{array}
	\end{equation}
where we used that $u$ is bounded from above and $|\nabla u| \le C |\nabla r|/r \le C'/r$. Hence, $\hat{} \ $ induces an inclusion 
	\[
	\hat{} \ \ : \ \ \mathscr{H} \longrightarrow H^1(\mathbb{S}_M, \hat g)
	\]
into a subspace $\hat{\mathscr{H}}$ of the Sobolev space $H^1$ of spinors for $(M, \hat g)$. The manifold $(M, \hat g)$ has bounded geometry in the sense of \cite[Definition 2.2]{GN}, so the trace theorem in \cite[Theorem 3.7]{GN} implies that the restriction 
	\[
	\mathrm{R} \ \ : \ \ \psi \in \mathscr{H} \quad \longmapsto \quad \hat \psi_{|\Sigma} \in H^{\frac{1}{2}}(\mathbb S_M|_\Sigma, \hat g)
	\]
is a bounded operator. Furthermore, we claim that the functional 
	\[
	(\psi, \xi) \in \Gamma_c(\mathbb S_M)\times \Gamma_c(\mathbb S_M) \longmapsto \int_\Sigma \langle (H_g + \mathcal{U})\psi, \xi \rangle d\Sigma
	\]
extends continuously on $\mathscr{H} \times \mathscr{H}$, so in particular \eqref{ide_spectral} extends by density to $\psi \in \mathscr{H}_\pm$. Using $|h| \le C r^{-1-\tau}$, we compute
	\[
	\begin{array}{lcl}
	\disp \left| \int_\Sigma \langle (H + \mathcal{U})\psi, \xi \rangle d\Sigma \right| & \le & \disp C \int_\Sigma r^{-1-\tau} |\langle \psi, \xi \rangle| d\Sigma = C \int_\Sigma r^{-1-\tau} |\langle \hat \psi, \hat \xi \rangle_{\hat g}| r^{2 \frac{n-1}{n}} d\Sigma_{\hat g} \\[0.4cm]
	& \le & \disp C\int_\Sigma r^{- \frac{(n-2)^2}{2n}} | \langle \hat \psi, \hat \xi \rangle_{\hat g}| d\Sigma_{\hat g},
	\end{array}
	\] 	
where, in the last inequality, we used that $\tau > (n-2)/2$. By \cite[Lemma 3.8]{GN}, the $L^2$ product 
	\[
	(\hat \psi, \hat \xi) \mapsto \int_\Sigma \langle \hat \psi, \hat \xi \rangle_{\hat g} d \Sigma_{\hat g}
	\]
on $\Gamma_c(\mathbb S_M, \hat g)$ extends continuously to $H^1(\mathbb S_M, \hat g) \times H^{-1}(\mathbb S_M, \hat g)$, hence to $H^1(\mathbb S_M, \hat g) \times H^{1}(\mathbb S_M, \hat g)$ in view of \cite[Remark 3.6]{GN}. Concluding, from \eqref{eq_boundbordo} we get
	\begin{equation}\label{eq_bounconti}
	\disp \left| \int_\Sigma \langle (H + \mathcal{U})\psi, \xi \rangle d\Sigma \right| \le C \|\psi\|_{\mathscr{H}} \|\xi\|_{\mathscr{H}},
	\end{equation}
as claimed.

Denote by $\hat \varrho$ the unit normal vector to $\Sigma \hookrightarrow (M,\hat g)$ corresponding to $\varrho$, and define the MIT boundary operators 
	\[ 
	\begin{array}{l}
	\disp \mathscr{B}_\pm :  H^{\frac{1}{2}}(\mathbb{S}_M|_\Sigma, \hat g) \to H^{\frac{1}{2}}(\mathbb S_M|_\Sigma, \hat g), \qquad \eta \mapsto \hat{\omega} \eta \mp \eta, \\[0.3cm] 
	\mathscr{K}_\pm : = \mathscr{B}_\pm \circ \mathrm{R}, \qquad \mathscr{K}_\pm : \mathscr{H} \to  H^{\frac{1}{2}}(\mathbb{S}_M|_\Sigma, \hat g),
	\end{array}
	\]
with $\hat \omega \eta : = i \hat \varrho \, \hat{\cdot} \, \eta$. Note that $\mathscr{B}_\pm$ and $\mathscr{K}_\pm$ are continuous. The commutativity $\hat \varrho \, \hat{\cdot} \, \hat \psi = \widehat{\varrho \cdot \psi}$ guarantees that $\mathscr{H}_\pm = \ker \mathscr{K}_\pm$, hence because of \eqref{eq_coerci} the operator
	\[
	(\overline{\mathcal{D}}, \mathscr{K}_\pm) \ \ : \ \ \mathscr{H} \quad \longmapsto L^2(\mathbb S_M) \times H^{\frac{1}{2}}(\mathbb{S}_M|_\Sigma, \hat g)		
	\]
is injective.  
We then proceed as in \cite{bartnikchrusciel}. By \eqref{eq_coerci}, the quadratic form 
	\[
	(\psi, \xi) \in \mathscr{H}_\pm \times \mathscr{H}_\pm \quad \longmapsto (\overline{\mathcal{D}} \psi, \overline{\mathcal{D}} \xi)_2 \in \mathbb{C}
	\]	
is coercive, so for given $\eta \in L^2(\mathbb S_M)$ and associated continuous functional $\psi \in \mathscr{H}_\pm \mapsto (\overline{\mathcal{D}} \psi, \eta)_2$, there exists $\varphi \in \mathscr{H}_\pm$ such that
	\[
	(\overline{\mathcal{D}} \psi, \overline{\mathcal{D}} \varphi - \eta)_2 = 0 \qquad \forall \, \psi \in \mathscr{H}_\pm,
	\]  
namely, $\Psi : = \overline{\mathcal{D}} \varphi - \eta \in L^2(\mathbb S_M)$ weakly solves the adjoint Dirac equation $(\overline{\mathcal{D}} \psi, \Psi)_2= 0$ for every $\psi \in \mathscr{H}_\pm$.  The identity in the proof of Proposition \ref{sefaddir}, that can be rewritten as 
	\[
	\begin{array}{lcl}
	\disp (\overline{\mathcal{D}} \psi, \xi)_2 - (\psi, \overline{\mathcal{D}} \xi)_2 & = & \disp \int_\Sigma \langle \varrho \cdot \psi, \xi \rangle d\Sigma \\[0.4cm]
	& = & \disp \frac{1}{2} \int_\Sigma \big[ \langle \varrho\cdot \psi, \mathscr{K}_- \xi \rangle - \langle \varrho\cdot\mathscr{K}_+ \psi, \xi \rangle\big] d\Sigma 	\qquad \forall \, \psi,\xi \in \Gamma_c(\mathbb S_M),
	\end{array}
	\]	
shows that the formal adjoint of $(\overline{\mathcal{D}}, \mathscr{K}_\pm)$ is $(\overline{\mathcal{D}}, \mathscr{K}_\mp)$. In particular,	
	\[
	\overline{\mathcal{D}} \Psi = 0 \ \  \text{ in } \, L^2_\loc(\mathbb S_M)\qquad
	\text{and}\qquad \mathscr{K}_\mp \Psi = 0 \ \ \text{ in } \, H^{\frac{1}{2}}_\loc(\mathbb{S}_M|_\Sigma). 
	\]
Let $\{\eta_k\}$ be the cut-offs in \eqref{def_etaR} with indices $R = k \in \mathbb N$, and define $\Psi_k = \eta_k \Psi \in H^1_c(\mathbb S_M) \cap \mathscr{H}_\mp$. By \eqref{eq_coerci},
	\[
	\begin{array}{lcl}
	\disp C'\|\Psi_{k}-\Psi_j\|^2_{\mathscr{H}} & \le & \disp \|\overline{\mathcal{D}}(\Psi_k-\Psi_j)\|^2_2 = C\int_M \big[|\nabla \eta_k|^2|\Psi|^2 + |\nabla \eta_j|^2|\Psi|^2\big] dM \\[0.4cm]
	& \to & 0 \qquad \text{as } \, j,k \to \infty,  	
	\end{array}
	\]
thus $\{\Psi_k\}$ is a Cauchy hence convergent sequence in $\mathscr{H}$. Concluding, $\Psi \in \mathscr{H}_\mp$, hence $\Psi \in \ker (\overline{\mathcal{D}}, \mathscr{K}_\mp) = 0$, so that $\varphi$ solves \eqref{eq:phi}.	In view of the higher regularity estimates in \cite{bartnikchrusciel} (cf. Theorems 3.8, 6.6 and Remark 6.7 therein), $\varphi \in H^{2}_\loc(\mathbb{S}_M)$ whenever $\eta \in H^1_\loc(\mathbb S_M)$.
\end{proof}

We proceed by choosing a non-trivial {\em parallel} spinor $\phi\in \Gamma(\mathbb S_{\mathbb R^n_{+,r_0}})$ satisfying $i \partial_{x_n} \cdot \phi = \pm \phi$, and transplant it to a spinor $\hat \phi\in\mathbb{S}_{M_{\rm{ext}}}$  satisfying (\ref{mitcond}). We extend $\hat \phi$ to the rest of $\Sigma$ so that the boundary condition holds everywhere, and finally extend $\hat \phi$ to the rest of $M$ in an arbitrary manner.
It follows from 
(\ref{eq_nablaintrinsic}) and  (\ref{spingauss}) that  
\[
\overline\nabla_{e_i}\hat \phi=\partial_{x_i} \hat \phi+\frac{1}{4}\Gamma_{ij}^k \partial_{x_j}\cdot \partial_{x_k}\cdot \hat \phi-\frac{1}{2}h_{ij}e_0\cdot e_j\cdot \hat\phi.
\]
Since $\partial_{x_i} \hat\phi=0$, by (\ref{orthon}) we then have
$\overline\nabla_{e_i} \hat\phi=O(|x|^{-\tau-1})$,
that is, $\overline\nabla \hat \phi\in L^2(\mathbb S_M)$. By Proposition \ref{surjdir}
we can find a spinor $\varphi\in \mathscr{H} \cap H^2_\loc(M)$ such that $\overline{\mathcal D}\varphi=-\overline{\mathcal D} \hat\phi$ and satisfying (\ref{mitcond}) along $\Sigma$. We define
\begin{equation}\label{def:psi}
\psi= \hat \phi+\varphi\in H^2_\loc(\mathbb S_M).
\end{equation}
Thus, $\psi$ is harmonic ($\overline{\mathcal D}\psi=0$), satisfies (\ref{mitcond}) along $\Sigma$, and asymptotes $\hat \phi$ at infinity in the sense $\psi- \hat\phi\in \mathscr{H}$.

The next result gives a nice extension of Witten's celebrated formula for the energy-momentum vector of an asymptotically flat initial data set in the presence of a non-compact boundary. More precisely, it is the spacetime version of \cite[Theorem 5.2]{ABdL}.  

\begin{theorem}\label{genwit}
	If the asymptotically flat initial data set $(M,g,h,\Sigma)$ satifies the DECs (\ref{intdec2}) and (\ref{bddectan2}) and 
	 $\psi$ is the harmonic spinor in (\ref{def:psi}) then 
	\begin{eqnarray}\label{genwit2} 
	\frac{1}{4}\left(E|\phi|^2-\langle \phi,P_A \partial_{x_0}\cdot \partial_{x_A}\cdot\phi\rangle \right) 
	& = & \int_{M}(|\overline\nabla\psi|^2+\langle \mathcal R\psi,\psi\rangle)d{\rm vol}\nonumber\\
	& & +\frac{1}{2}\int_{\Sigma}\langle \left(H_g+ \mathcal U\right)\psi,\psi\rangle d\Sigma.  	
	\end{eqnarray}
	\end{theorem}
\begin{proof}
From  (\ref{decom2}) and (\ref{intpartf0}) 
we get
	\begin{eqnarray*}
\int_{\Omega_r}(|\overline\nabla\psi|^2+\langle \mathcal R\psi,\psi\rangle)dM & = & \int_{S^{n-1}_{r,+}}\langle(\nabla_{e_i}+e_i\cdot{\mathcal D}) \psi,\psi\rangle e_i\righthalfcup dM\\ 
& & \quad -\frac{1}{2}\int_{S^{n-1}_{r,+}}\pi_{ij}\langle e_0\cdot e_i\cdot\psi,\psi\rangle e_j\righthalfcup dM\\
& & \quad\quad  -\frac{1}{2}\int_{\Sigma_r}\langle (H_g+\mathcal U)\psi,\psi\rangle d\Sigma.
\end{eqnarray*}	
First, the computation in \cite[Section 5.2]{ABdL} shows that 
\[
\lim_{r\to +\infty}\int_{S^{n-1}_{r,+}}\langle(\nabla_{e_i}+e_i\cdot{\mathcal D}) \psi,\psi\rangle e_i\righthalfcup dM=\frac{1}{4}E|\phi|^2.
\]
Also, 
\begin{equation}\label{eq_bordo1}
\begin{array}{lcl}
\disp \lim_{r\to +\infty} \int_{S^{n-1}_{r,+}}\pi_{ij}\langle e_0\cdot e_i\cdot\psi,\psi\rangle e_j\righthalfcup dM 
& = & \disp \lim_{r\to +\infty} \int_{S^{n-1}_{r,+}}\pi_{ij}\langle e_0\cdot e_i\cdot \hat \phi, \hat \phi\rangle e_j\righthalfcup dM \\
&= & \disp \lim_{r\to +\infty} \int_{S^{n-1}_{r,+}}\pi(\partial_{x_i},\mu)\langle \partial_{x_0} \cdot \partial_{x_i} \cdot \phi, \phi\rangle  dS^{n-1}_{r,+}\\
& = & \disp \lim_{r\to +\infty} \int_{S^{n-1}_{r,+}}\pi(\partial_{x_A},\mu)\langle \partial_{x_0}\cdot \partial_{x_A }\cdot \phi, \phi\rangle  dS^{n-1}_{r,+}\\
& = & \disp \frac{1}{2}\langle P_A\partial_{x_0}\cdot \partial_{x_A}\cdot\phi,\phi\rangle,
\end{array}
\end{equation}
where we used (\ref{momdef}) together with the fact that, by Remark \ref{rem_MIT} and since $\phi\in\Gamma(\mathbb S_{\mathbb R^n_{+,r_0}})$ is constant, $\langle \partial_{x_0} \cdot \partial_{x_n} \cdot\phi,\phi\rangle=0$ on the entire $\mathbb R^n_+$. 
\end{proof}

\begin{remark}
 The argument leading to the first identity in \eqref{eq_bordo1} needs a little justification, since a-priori  $|\psi - \hat \phi| \in \mathscr{H}$, in particular we only know that
	\[
	\lim_{j \to \infty} r_j \int_{S^{n-1}_{r_j,+}} \frac{|\psi - \hat \phi|^2}{r_j^2} dS^{n-1}_{r_j,+} = 0
	\]
for some diverging sequence $\{r_j\}$. Using 
	\[
	\big|\langle e_0\cdot e_i\cdot\psi,\psi\rangle - \langle e_0\cdot e_i\cdot \hat \phi, \hat \phi\rangle \big| \le \disp \big[|\psi| + |\hat \phi|\big]|\psi - \hat \phi| \le \disp \big[2|\hat\phi| + |\psi - \hat\phi| \big] |\psi - \hat\phi| \,,
	\]
we readily deduce that 
	\[
	\begin{array}{l}
	\disp \left| \int_{S^{n-1}_{r_j,+}}\pi_{ij} \big[ \langle e_0\cdot e_i\cdot\psi,\psi\rangle - \langle e_0\cdot e_i\cdot\hat\phi,\hat\phi\rangle \big] e_j\righthalfcup dM\right| \\[0.5cm]
\disp \quad \le C r_j^{-\tau-1} \int_{S^{n-1}_{r_j,+}} \Big[ 2 \|\hat\phi\|_\infty |\psi-\hat \phi| + |\psi-\hat \phi|^2\Big] dS^{n-1}_{r_j,+} \\[0.5cm]
\disp \quad \le C r_j^{-\tau-1} \left\{ 2\|\hat\phi\|_\infty \sqrt{|S^{n-1}_{r_j,+}|r_j} \left( \int_{S^{n-1}_{r_j,+}} \frac{|\psi-\hat\phi|^2}{r_j} dS^{n-1}_{r_j,+}\right)^{\frac{1}{2}} + r_j \int_{S^{n-1}_{r_j,+}} \frac{|\psi-\hat \phi|^2}{r_j} dS^{n-1}_{r_j,+} \right\} \to 0 
	\end{array}
	\]
as $j \to \infty$, where we used the full strength of the inequality $\tau > (n-2)/2$. This is enough to conclude that the claimed equality holds if the limit is evaluated along the sequence $\{r_j\}$. In particular, each of the limits in \eqref{eq_bordo1} should be read as a sequential limit along $\{r_j\}$, which nevertheless suffices to guarantee \eqref{genwit2}.  
\end{remark}

In order to complete the proof of  Theorem \ref{main}, we make one last assumption on the parallel spinor $\phi$ used in the construction of $\psi$ in (\ref{def:psi}). As in Remark \ref{decbdr}, one checks that the operator $\mathcal T=P_A \partial_{x_0}\cdot \partial_{x_A}\cdot$ has $\pm|P|$ as eingenvalues. Also, it satisfies $\mathcal T(i\partial_{x_n})=(i\partial_{x_n})\mathcal T$. In particular, $\mathcal T$ and $i\partial_{x_n} \cdot$ are simultaneously diagonalizable. Thus we may choose $\phi$ constant (parallel) in $\R^n_{+,r_0}$ satisfying both $\mathcal T\phi=|P|\phi$ and one of the MIT bag boundary conditions $i\partial_{x_n} \cdot \phi=\pm\phi$.
Using this $\phi$ in (\ref{genwit2}) we get 
\begin{equation}\label{eq_important}
\frac{1}{4}\left(E-|P|\right)|\phi|^2 = \disp \int_{M}(|\overline\nabla\psi|^2+\langle \mathcal R\psi,\psi\rangle)dM + \frac{1}{2}\int_{\Sigma}\langle \left(H_g+ \mathcal U\right)\psi,\psi\rangle d\Sigma.  	
\end{equation}
Since the right-hand side is nonnegative by Remarks \ref{claim} and \ref{decbdr}, we obtain the mass inequality 
\begin{equation}\label{mass:ineq}
E\geq|P|.
\end{equation}
Suppose next that $E = |P|$, so there exists $\psi \in H^2_\loc(\mathbb S_M)$ satisfying 
\begin{equation}\label{proprie_psipar}
\begin{array}{ll}
\disp \overline{\nabla} \psi = 0, \quad \langle \mathcal R\psi,\psi\rangle \equiv 0 & \qquad \text{on } M, \\[0.2cm]
i e_n \cdot \psi = \pm \psi, \quad \langle (H_g + \mathcal{U})\psi, \psi \rangle = 0, \quad (e_n \cdot \psi, \psi) = 0  & \qquad \text{on } \, \Sigma,
\end{array}
\end{equation}
and that \eqref{decay_rhoJ_forrigid} is in force. By bootstrapping, $\overline{\nabla} \psi = 0$ readily implies that $\psi \in C^{2,\alpha}_\loc (\mathbb S_M)$. Our first goal is to prove that $E = |P| = 0$. To do so, we avail of the following identity in \cite{dLGM}:
\begin{equation}\label{energy_LGM}
E = d_n \cdot \lim_{r \to \infty} \left[ \int_{\mathbb S^{n-1}_{r,+}} G( r \partial_r, \mu) d \mathbb{S}^{n-1}_{r,+} - \int_{\mathbb{S}^{n-2}_r} N(r \partial_r, \vartheta) d \mathbb{S}^{n-2}_r\right],
\end{equation}
where $G$ is the Einstein tensor of $M$, $N$ is the first Newton tensor of $\Sigma \hookrightarrow M$ in the direction $\varrho= e_n$, and $d_n = \frac{2}{2-n}$.
We consider the lapse-shift pair associated to $\psi:$
$$
V = \langle \psi,\psi\rangle, \qquad W = \sum_{j=1}^n(e_j \cdot \psi, \psi)e_j = W^j e_j. 
$$
In a frame $\{e_i\}$ that is $\nabla$-normal at a given point,
\begin{equation}\label{ar_notbad}
\begin{array}{rlcl}
(i) & V_i & = & \disp e_i( e_0 \cdot \psi, \psi) = ( \overline{\nabla}_{e_i} e_0 \cdot \psi, \psi) = h_{ik}(e_k \cdot \psi, \psi) = h_{ik} W^k \\[0.3cm]
(ii) & \disp \nabla_{e_i} W & = & \disp e_i(e_j \cdot \psi, \psi)e_j = (\overline{\nabla}_{e_i} e_j \cdot \psi, \psi)e_j = h_{ij}(e_0 \cdot \psi, \psi)e_j = h_{ij} V e_j,
\end{array}
\end{equation}
so 
\begin{equation}\label{eq_VW}
\mathscr{L}_W g - 2Vh = 0, \qquad d(V^2-|W|^2) = 0 \qquad \text{on } \, M.
\end{equation}
Furthermore, $W^n = (e_n \cdot \psi, \psi) = 0$.\\[0.2cm] 
\noindent \textbf{Claim 1: } We have $V^2 \ge |W|^2$.\\[0.2cm] 
\emph{Proof:} The claim is obvious at points where $\psi = 0$. On the other hand, if $|\psi|^2 > 0$, then $\{e_j \cdot \psi\}$ is an orthogonal set of spinors for $\langle \, , \, \rangle$, with $|e_j \cdot \psi|^2 = V$ for every $j$ (in fact, $\mathfrak{Re} \langle e_i \cdot \eta, e_j \cdot \eta \rangle = \delta_{ij} |\eta|^2,\,\forall \, \eta \in \Gamma(\mathbb{S}_M)$). Thus
\begin{equation}\label{eq_causal}
|W|^2 = \sum_j (e_j \cdot \psi, \psi)^2 = \sum_j \langle e_j \cdot \psi, e_0 \cdot \psi \rangle^2 \le V|e_0 \cdot \psi|^2 = V^2.
\end{equation}
\noindent \textbf{Claim 2:} If $E=|P|$, then $V>0$ on $M$ and 
	\[
	\Sigma \hookrightarrow M \quad \text{is totally geodesic,} \qquad \pi_{An}=0 \quad \text{on } \, \Sigma, \ \ \forall \, A \in \{1, \ldots, n-1\}.
	\]
Consequently, $e_n(V) =0$ on $\Sigma$.\\[0.2cm]
\emph{Proof: } We use \eqref{ar_notbad} and \eqref{eq_causal} to compute
	\begin{equation}\label{thekey}
	\begin{array}{lcl}
	\Delta V & = & \disp (h_{ik}W^k)_i = |h|^2 V + h_{ik,i}W^k \\[0.2cm]
	& \le & \disp |h|^2 V + |\nabla h| |W| \le \big(|h|^2 + |\nabla h|\big) V \qquad \text{on } \, M.
	\end{array}
	\end{equation}
By the strong maximum principle, we deduce that either $V \equiv 0$ or $V>0$ on $\mathrm{Int}\,M$. The first case cannot occur, because otherwise $\psi \equiv 0$ and thus $\psi - \hat \phi \not\in \mathscr{H}$. Next, we differentiate the equality $i e_n \cdot \psi = \pm \psi$ and use $\overline{\nabla} \psi =0$ to get
\begin{equation}\label{eq_rig_boundary}
\begin{array}{lcl}
0 & = & \disp \overline{\nabla}_{e_A} (e_n \cdot \psi) = (\overline{\nabla}_{e_A} e_n) \cdot \psi \\[0.2cm]
& = & \disp -b_{AB} e_B \cdot \psi - \bar g(\overline\nabla_{e_A} e_n,e_0) e_0 \cdot \psi = (-b_{AB}e_B + \pi_{An}e_0) \cdot \psi.
\end{array}
\end{equation}
where $b_{AB}$ is the second fundamental form of $\Sigma \hookrightarrow M$ in the inward direction $e_n$ (so $H_g = \mathrm{tr}_g b$). Making the product $\metric$ with, respectively, $e_B \cdot \psi$ and $e_0 \cdot \psi$, we obtain the system
	\begin{equation}\label{boundary_1}
	\left\{ \begin{array}{ll}
	(i) & -b_{AB} V + \pi_{An}W^B = 0 \qquad \forall \, A,B \\[0.2cm]
	(ii) & -b_{AB} W^B + \pi_{An}V = 0 \qquad \forall \, A.
	\end{array}\right.
	\end{equation}
Suppose that $V(x) =0$ for some $x \in \Sigma$. From $(i)$, we then deduce $\pi_{An}W^A = 0$ at $x$, and in view of the identity $W^n=0$ on $\Sigma$ we infer 
	\[
	e_n(V) = h_{nj}W^j = \pi_{nA}W^A = 0 \qquad \text{at } \, x,
	\]
contradicting the Hopf maximum principle. This shows that $V>0$ on $\Sigma$, hence on the entire $M$. Differentiating the pointwise inequality $\langle (H_g+\mathcal{U}) (\psi+ t\eta),\psi + t\eta)\rangle \ge 0$ at $t=0$, because of \eqref{proprie_psipar} we get $\mathfrak{Re} \langle (H_g + \mathcal{U})\psi,\eta \rangle \equiv 0$ on $\Sigma$ for each $\eta \in C^2_\loc (\mathbb S_M)$. Using as a test spinor, respectively, $\eta = \psi$ and $\eta = e_0 \cdot e_B \cdot \psi$ we get	
	\begin{equation}\label{boundary_2}
	\left\{ \begin{array}{ll}
	(iii) & 0 = \mathfrak{Re} \langle (H_g + \mathcal{U})\psi, \psi \rangle = H_g V + \pi_{An}W^A \\[0.2cm]
	(iv) & 0 = \mathfrak{Re} \langle (H_g + \mathcal{U})\psi, e_0 \cdot e_B \cdot \psi \rangle = H_g W^B + \pi_{Bn} V.
	\end{array}\right.
	\end{equation}
Tracing $(i)$ in $A,B$ and comparing with $(iii)$ we deduce $H_g V = 0$, hence $H_g=0$ on $\Sigma$. Plugging into $(iv)$ gives $\pi_{Bn} =0$, in particular $e_n(V) = \pi_{Bn}W^B =0$. By $(i)$, we conclude $b_{AB}=0$ on $\Sigma$.  \\[0.2cm]
\noindent \textbf{Claim 3: } If $E=|P|$, then
\begin{equation}\label{claim_1} 
\rho_0 V = J(W), \qquad \rho_0 W = V J^\sharp, \qquad \rho_0 = |J| \qquad \text{on } \, M. 
\end{equation}
Furthermore, $V^2 \equiv |W|^2$ unless $\rho_0 = |J| \equiv 0$ on $M$.\\[0.2cm]
\emph{Proof: } We differentiate the pointwise inequality $\langle \mathcal R (\psi+ t\eta),\psi + t\eta)\rangle \ge 0$ at $t=0$ and use the second identity in \eqref{proprie_psipar} to deduce $\mathfrak{Re} \langle \mathcal R \psi,\eta \rangle \equiv 0$ on $M$ for each $\eta \in C^2_\loc (\mathbb S_M)$. Using as a test spinor, respectively, $\eta = \psi$ and $\eta = e_0 \cdot e_k \cdot \psi$ we get
$$
\begin{array}{lcl}
0 & = & \disp 2 \mathfrak{Re} \langle \mathcal R \psi, \psi \rangle = \disp \rho_0 V + J_i \mathfrak{Re} \langle e_i \cdot e_0 \cdot \psi, \psi \rangle = \rho_0 V - J_iW^i \\[0.2cm]
0 & = & \disp 2 \mathfrak{Re} \langle \mathcal R \psi, e_0 \cdot e_k \cdot \psi \rangle = \disp \rho_0 (\psi, e_k \cdot \psi) + J_i \mathfrak{Re} \langle e_i \cdot e_0 \cdot \psi, e_0 \cdot e_k \cdot \psi \rangle \\[0.2cm]
& = & \disp \rho_0 W_k - J_i \mathfrak{Re} \langle e_i \cdot \psi, e_k \cdot \psi \rangle = \rho_0 W_k - J_k V, 
\end{array}
$$
proving the first two identities in \eqref{claim_1}. Multiplying the second one in \eqref{claim_1} by $J^\sharp$ we deduce that
\[
\disp \rho_0^2 V = \rho_0 J(W) = V |J|^2,
\]
and $\rho_0 = |J|$ follows since $V>0$ on $M$. Similarly, multiplying that identity by $W$ we get $\rho_0 |W|^2 = VJ(W) = \rho_0 V^2$. If $\rho_0 \not \equiv 0$, then $V^2 \equiv |W|^2$ on the entire $M$, in view of the constancy of $V^2-|W|^2$. \\[0.2cm]
Let $(M',g)$ consists of two copies of $M$ glued along the totally geodesic boundary $\Sigma$. Then, $M'$ is asymptotically flat and the metric $g$ extended by reflection across $\Sigma$ is $C^{2,1}_\loc$. \\[0.2cm]
\noindent \textbf{Claim 4: } If $E = |P|$, then $g -\delta \in C^{2,\alpha}_{-\tau}(M')$ and $h \in C^{1,\alpha}_{-\tau-1}(M')$, where $h_{ij}$ is extended by reflection along $\widehat\Sigma$.\\[0.2cm]

\noindent \emph{Proof. } We consider the manifold 
$$
\widehat M = M \times \R
$$
endowed with the following \emph{Riemannian} metric $\widehat g$. Let $\{x_i\}$ be local smooth coordinates around a point $p \in \widehat M$. We define
\begin{equation}\label{eq_riemKilling}
\widehat g = V^2 du^2 + g_{ij} (dx^i + W^i d u) \otimes ( dx^j + W^j du),
\end{equation}
where $W = W^i \partial_{x_j}$. 
Hereafter, the functions $g_{ij}$ and $W^i$
are meant to be functions on $\widehat M$ by pre-composing with the projection $\widehat M \to M$ onto the first factor. 
The manifold $(\widehat M,\widehat g)$ is a Riemannian analogue of the Killing development of $(M,g)$, described in \cite{CM} and recalled in the Appendix.
Direct calculations show that 
$$
\partial_\varrho g(\partial_u, \partial_u)
=\partial_\varrho g(\partial_u, \partial_{x_A})
=\partial_\varrho g(\partial_{x_B}, \partial_{x_A})=0,
\qquad\text{on}\:\:\widehat\Sigma:=\Sigma\times\mathbb R, 
$$
which implies that $\widehat\Sigma$ is totally geodesic with respect to $\widehat g$.

We consider again the double $(\widehat M', \widehat g)$ of $(\widehat M, \widehat g)$ along $\widehat\Sigma$. Here, $\widehat g$ extends by reflection across $\widehat\Sigma$ with regularity $\widehat g\in C^{2,1}_{loc}(\widehat M')$ as $\widehat\Sigma$ is totally geodesic. As in its Lorentzian counterpart, the second fundamental form of each slice $\{u={\rm {const}}\}\subset\widehat M'$ is given by $h_{ij}$, so we conclude that $h_{ij}\in C^{1,\alpha}_{-\tau-1}(M')$, thus proving the claim.

\bigskip
We now conclude that $E = |P| = 0$. First, using Claims $2$ and $3$, and since $\rho_0$ and $J$ are clearly continuous across $\Sigma$, we get $\rho_0, J \in C^{0,\alpha}$ across $\Sigma$. Consequently, by \eqref{decay_rhoJ_forrigid},
	\[
	\rho_0, J \in C^{0,\alpha}_{-n-\varepsilon}(M').
	\]
The data $(g,h,\rho_0,J)$ on $M'$ therefore satisfies the requirements to apply the rigidity result in \cite[Theorem 2]{HL}. Denote with $(E',P')$ the energy-momentum vector of $M'$. From \eqref{energy_LGM}, the fact that $\Sigma$ is totally geodesic, and the identity corresponding to \eqref{energy_LGM} in the boundaryless case (due to \cite{AH, Hu, MT, He2}), we get
\begin{equation}\label{energy_LGM_semborda}
E = d_n \cdot \lim_{r \to \infty} \int_{\mathbb S^{n-1}_{r,+}} G( r \partial_r, \mu) d \mathbb{S}^{n-1}_{r,+} = \frac{1}{2}E'.
\end{equation}
On the other hand, by symmetry, the component $P'_n$ of the momentum vector of $M'$ satisfies
	\[
	P_n' = 2 \lim_{r \to \infty} \int_{\mathbb S^{n-1}_{r}} \pi(\partial_{x_n}, \mu) d \mathbb{S}^{n-1}_{r} = 0,
	\]
and evidently $P_A' = 2 P_A$ for $1 \le A \le n-1$. Hence, $E' = |P'|$, and \cite[Theorem 2]{HL} guarantees the vanishing of $E'$ and $|P'|$.\\
Having shown that $E = |P| = 0$, choose a basis of parallel spinors $\{\phi_m\}$ for $\mathbb S_{\R^n_+}$, with each $\phi_m$ satisfying (\ref{mitcond}). From \eqref{genwit2}, the corresponding harmonic spinors $\psi_m \in \Gamma(\mathbb S_M)$ satisfy
	\begin{eqnarray*}
	\int_{M}(|\overline\nabla\psi_m|^2+\langle \mathcal R\psi_m,\psi_m\rangle)d M \nonumber +\frac{1}{2}\int_{\Sigma}\langle \left(H_g+ \mathcal U\right)\psi_m,\psi_m\rangle d\Sigma = 0 \qquad \forall \, m,	
\end{eqnarray*}
and are therefore parallel, in particular pointwise linearly independent everywhere. Combining this with
\[
\begin{array}{lcl}
0 & = & \overline{R}_{e_i,e_j}\psi_m := \disp \left(\overline\nabla_{e_i}\overline\nabla_{e_j}-\overline\nabla_{e_j}\overline\nabla_{e_i}-\overline\nabla_{[e_i,e_j]}\right)\psi_m \\[0.2cm]
& = & \disp \left\{ - \frac{1}{4} \Big[ R_{ijks} + h_{ik}h_{js}-h_{is}h_{jk} \Big]e_k \cdot e_s \cdot + \frac{1}{2}[h_{kj,i}-h_{ki,j}] e_k \cdot e_0 \cdot \right\} \psi_m,
\end{array}
\]  
we deduce 
\[
\overline{R}_{ijks} : = R_{ijks} + h_{ik}h_{js}-h_{is}h_{jk} \equiv 0, \qquad  \overline{R}_{k0ij} : = h_{kj,i}-h_{ki,j} \equiv 0.
\]
From the above two identities it readily follows that the Killing development $\bar M$ of $M$ (as defined in \cite{BC,CM}, see also the Appendix) is flat, and that $\partial \bar M \hookrightarrow \bar M$ is totally geodesic. Moreover, $M$ meets $\partial \bar M$ orthogonally along $\Sigma$. Since $M$ is asymptotically flat, the only possibility is that $\bar M = \mathbb{L}^{n,1}_+$. This completes the proof of Theorem \ref{main}.

\begin{remark}\label{disc:bound:cond}
	Nowhere in the argument leading to Theorem \ref{main} we assume that the noncompact boundary is connected, so we can allow for  finitely many {\em compact} components as part of $\Sigma$ without compromising the final result, including the rigidity statement. We may additionally envisage a situation in which an extra finite family of {\em inner} compact boundary components, say $\{\Gamma_k\}_{k=1}^l$, should be viewed as (past or future) trapped hypersurfaces. Along these boundary components we must instead impose chirality boundary conditions (as in Definition \ref{chirbd} below) so that the extra term appearing in the right-hand side of (\ref{intpartf0}) becomes
	\[
	 -\frac{1}{2}\sum_k\int_{\Gamma_k}(H_g\pm\pi_{nn})|\psi|^2 d\Gamma_k;
	\]   
	see Remark \ref{rem_chiral} below for a justification of the cancellation leading to this boundary contribution.
	Thus, in order to
	make sure that this term has the right sign needed to carry out the reasoning we must
	impose the usual trapped condition $H_g\pm\pi_{nn}\geq 0$ along the inner boundaries. Since it is well-known that the existence of such trapped inner boundaries foretells, under suitable global conditions, the existence of a black hole region in the Cauchy development of the given initial data set, this arguments extends to our setting well-known black hole positive mass theorems previously established in case $\Sigma$ is empty \cite{GHHP, He1, XD}. Notice however that the rigidity statement does not seem to persist here, and we are led to infer that the presence of trapped surfaces  
	forces the corresponding energy-moment vector to be necessarily time-like. This clearly suggests that a Penrose-type inequality might hold in this setting, a proposition that, at least in the time-symmetric case with $n=3$, has been  recently  confirmed \cite{Ko}.  
\end{remark}


\section{The asymptotically hyperbolic case}\label{secthyp}

In this section we prove Theorems \ref{maintheohyp} and \ref{maintheoohyp2}. We begin by proving Theorem \ref{maintheohyp} which is inspired by \cite{Ma}; see also \cite{CMT}.
Let $(M,g,h,\Sigma)$ be an initial data set with $(M,g)\hookrightarrow(\overline M,\overline g)$ as in the statement of Theorem \ref{maintheohyp}.
As in Section \ref{spinors}, over the spin slice $M$ we have both an extrinsic and an intrinsic description of the restricted spin bundle  $\mathbb S_M$. 
Thus, $\mathbb S_M$ comes endowed with the inner products $(\, ,\,)$ and $\langle\, ,\,\rangle$ and the connections $\overline \nabla$ and $\nabla$, which allow us to define the Dirac-Witten and the intrinsic Dirac operators $\overline {\mathcal D}$ and $\mathcal D$, respectively. We then define the {\em Killing connections} on $\mathbb S_M$ by
\[
{\overline {\nabla}}^{\pm}_X=\overline \nabla_X\pm\frac{{\rm i}}{2}X\cdot
\]
and the corresponding {\em Killing-Dirac-Witten operators} by
\[
{\overline {\mathcal D}}^{\pm}=e_i\cdot{\overline{\nabla}}^{\pm}_{e_i}\,. 
\] 
It is clear that 
\begin{equation}\label{eq:dirac:mod}
{\overline{\mathcal D}}^{\pm}=\overline {\mathcal D}\mp\frac{n{\rm i}}{2},
\end{equation}
which after (\ref{selfpoint}) gives
\begin{equation}\label{selfpoint2}
d\widehat\theta=(\langle{\overline {\mathcal D}}^\pm\phi,\xi\rangle-\langle \phi,{\overline {\mathcal D}}^\mp\xi\rangle)dM. 
\end{equation}

We now introduce the relevant boundary condition on spinors. Consider the chirality operator $Q=e_0\,\cdot:\Gamma(\mathbb S_M)\to \Gamma(\mathbb S_M)$. This is a (pointwise) selfadjoint involution, which is parallel (with respect to $\overline \nabla$) and anti-commutes with Clifford multiplication by tangent vectors to $M$. 
We then define the (pointwise) hermitean involution 
$$
\mathcal Q=Q\varrho\,\cdot=e_0\cdot \varrho\,\cdot,
$$ 
acting on spinors restricted to $\Sigma$.
\begin{definition}\label{chirbd}
	We say that $\psi\in\Gamma(\mathbb S_M)$ satisfies the {\em chirality boundary condition} if along $\Sigma$ it satisfies any of the identities
	\begin{equation}\label{chirbd2}
	\mathcal Q\psi=\pm\psi.
	\end{equation}
	\end{definition}  

\begin{proposition}\label{selfaddirpn}
	The operators $\overline{\mathcal D}^+$ and $\overline {\mathcal D}^-$ are formally adjoints to each other under any of the boundary conditions (\ref{chirbd2}).
	\end{proposition}

\begin{proof}
If $\phi$ and $\xi$ are compactly supported we have  
\[
\int_M\langle{\overline {\mathcal D}}^\pm\phi,\xi\rangle dM-\int_M\langle \phi,{\overline{\mathcal D}}^\mp\xi\rangle dM=\int_\Sigma\langle \varrho\cdot\phi,\xi\rangle d\Sigma. 
\]
However, if either $\mathcal Q\phi-\phi=0=\mathcal Q\xi-\xi$ or $\mathcal Q\phi+\phi=0=\mathcal Q\xi+\xi$ then it is easy to check that $\langle\varrho\cdot\phi,\xi\rangle=0$ on $\Sigma$.
\end{proof}
\begin{remark}\label{rem_chiral}{\rm 
	Note that if a spinor $\psi\in\Gamma(\mathbb S_{M})$ satisfies any of the chirality boundary conditions (\ref{chirbd2}) then $\langle e_0\cdot e_A\cdot\psi,\psi\rangle=0$ along $\Sigma$. Indeed, 
	\begin{eqnarray*}
		\langle e_0\cdot e_A\cdot\psi,\psi\rangle & = & 	\langle e_0\cdot e_A\cdot e_0\cdot e_n\cdot\psi,e_0\cdot e_n\cdot \psi\rangle\\
		& = & \langle  e_A\cdot e_0\cdot e_n\cdot\psi, e_n\cdot \psi\rangle\\
		& = & \langle e_n\cdot e_A\cdot e_0\cdot \psi, e_n\cdot \psi\rangle\\
			& = & \langle  e_A\cdot e_0\cdot \psi, \psi\rangle\\
			& = & -\langle  e_0\cdot e_A\cdot \psi, \psi\rangle.
		\end{eqnarray*} }
	\end{remark}

\begin{proposition}\label{moredir2}
	Given a spinor $\psi \in \Gamma(\mathbb{S}_M)$, define the $(n-1)$-forms
	\[
	\theta^+=\langle e_i\cdot\overline {\mathcal D}^+\psi,\psi\rangle e_i\righthalfcup dM 
	\quad\text{and}\qquad 
	\eta^+=\langle\overline \nabla_{e_i}^+\psi,\psi\rangle  e_i\righthalfcup dM.
	\]
	Then it holds
	\begin{equation}\label{12}
	d\theta^+=\left(\langle(\overline {\mathcal D}^+)^2\psi,\psi\rangle-|\overline {\mathcal D}^+\psi|^2\right)dM,
	\end{equation}
	and
	\begin{equation}\label{22}
	d\eta^+=\left(-\langle (\overline {\nabla}^+)^*\overline {\nabla}^+\psi,\psi\rangle+|\overline {\nabla}^+\psi|^2\right)dM.
	\end{equation}
\end{proposition}
\begin{proof}
	Straightforward computations similar to those of Proposition \ref{moredir}.
\end{proof}

We now combine this with the corresponding  Weitzenb\"ock formula, namely,
	\begin{equation}\label{wetz2}
	(\overline{\mathcal D}^+)^2=(\overline\nabla^+)^*\overline\nabla^++\mathcal W,
	\end{equation}
	where the symmetric endomorphism $\mathcal W$ is given by
	\[
	\mathcal W=\frac{1}{4}(R_{\overline g}+n(n-1)+2 {{\rm Ric}_{\overline g}}_{0\alpha}e_\alpha\cdot e_0\cdot).
	\]

\begin{remark}\label{claim2}{\rm 
	As in Remark \ref{claim}, we see that the DEC (\ref{intdec2}) with $\Lambda=-n(n-1)/2$ implies that $\mathcal W\geq 0$.}
\end{remark}

By putting together the results above we obtain a fundamental integration by parts formula. This is the analogue of Proposition \ref{put}.

\begin{proposition}\label{put2}
	If $\psi\in\Gamma(\mathbb S_M)$ and  $\Omega\subset M$ is compact then 
	\begin{equation}\label{intpart12}
	\int_{\Omega}(|\overline\nabla^+\psi|^2+\langle \mathcal W\psi,\psi\rangle-|\overline{\mathcal D}^+\psi|^2)dM=\int_{\partial\Omega}\langle(\overline\nabla^+_{e_i}+e_i\cdot\overline{\mathcal D}^+) \psi,\psi\rangle e_i\righthalfcup dM. 
	\end{equation}
\end{proposition}

As in Proposition \ref{intpart11}, 
we now specialize (\ref{intpart12}) to the case in which $\Omega=\Omega_r$, the compact region in a initial data set $(M,g,h,\Sigma)$ determined by the coordinate hemisphere $S_{r,+}^{n-1}$ in the asymptotic region; see Figure \ref{fig1} for a similar configuration.
\begin{proposition}
	With the notation above assume that $\psi\in\Gamma(\mathbb S_M)$ satisfies the boundary condition (\ref{chirbd2}) along $\Sigma$. Then
	\begin{eqnarray}\label{intpartf}
	\int_{\Omega_r}(|\overline\nabla^+\psi|^2+\langle \mathcal W\psi,\psi\rangle-|\overline{\mathcal D}^+\psi|^2)dM & = & \int_{S^{n-1}_{r,+}}\langle(\overline\nabla^+_{e_i}+e_i\cdot \overline{\mathcal{D}}^+) \psi,\psi\rangle e_i\righthalfcup dM\nonumber\\ 
	& & \quad -\frac{1}{2}\int_{\Sigma_r}\left(H_g\pm\pi_{nn}\right)|\psi|^2d\Sigma,
	\end{eqnarray}	
	where the $\pm$ sign agrees with the one in \eqref{chirbd2}.
\end{proposition}

\begin{proof}
	We first observe that, using (\ref{eq:dirac:mod}) and similarly to (\ref{decom2}), 
	\begin{equation}\label{dechyp1}
	\overline\nabla^+_{e_i}+e_i\cdot\overline{\mathcal D}^+ = \nabla_{e_i}+e_i\cdot{\mathcal D}-\frac{1}{2}\pi_{ij}e_0\cdot e_j\cdot+\frac{n-1}{2}{\rm \bf i}\,e_i\cdot,
	\end{equation}
	so that 
	\begin{eqnarray}\label{eq:rhs}
		\int_{\Sigma_r}\left\langle\left(\overline\nabla^+_{e_i}+e_i\cdot\overline{\mathcal D}^+\right)\psi,\psi\right\rangle e_i\righthalfcup dM
		& = & 	\int_{\Sigma_r}\left\langle\left(\nabla_{e_i}+e_i\cdot{\mathcal D}\right)\psi,\psi\right\rangle e_i\righthalfcup dM\notag\\
		& & \quad -\frac{1}{2}\int_{\Sigma_r}\pi_{ij}\langle e_0\cdot e_j\cdot\psi,\psi\rangle e_i\righthalfcup dM\notag\\
		& & \quad \quad +\frac{n-1}{2}\,{\rm \bf i}\int_{\Sigma_r}\langle \varrho\cdot \psi,\psi\rangle d\Sigma\notag\\
		& = & \int_{\Sigma_r}\left\langle \left(D^{\intercal}-\frac{H_g}{2}\right)\psi,\psi\right\rangle d\Sigma\notag\\
		& & \quad 	-\frac{1}{2}\int_{\Sigma_r}\pi(e_n,e_j)\langle e_0\cdot e_j\cdot\psi,\psi\rangle d\Sigma\notag\\
		& & \quad\quad  +\frac{n-1}{2}\,{\rm \bf i}\int_{\Sigma_r}\langle \varrho\cdot \psi,\psi\rangle d\Sigma.
		\end{eqnarray}
However, as in the proof of Propositions \ref{intpart11} and \ref{selfaddirpn}, and by  Remark \ref{rem_chiral}, the boundary condition (see \cite{AdL} for further discussions on this) implies that $\langle D^\intercal\psi,\psi\rangle=0$, $\langle\varrho\cdot\psi,\psi\rangle=0$ and $\langle e_0\cdot e_A\cdot\psi,\psi\rangle=0$. 	
So, the equation \eqref{eq:rhs} becomes 
\begin{align*}
\int_{\Sigma_r}\left\langle\left(\overline\nabla^+_{e_i}+e_i\cdot\overline{\mathcal D}^+\right)\psi,\psi\right\rangle e_i\righthalfcup dM
=-\frac{1}{2}\int_{\Sigma_r}\left(H_g|\psi|^2+\pi_{nn}\langle \mathcal Q\psi,\psi\rangle  \right)d\Sigma,
\end{align*}
from which the result follows.
\end{proof}

We now proceed with the proof of Theorem \ref{maintheohyp}. We start by picking a {\em Killing} spinor $\phi$ in the restricted reference spin bundle $\mathbb S_{\mathbb H^n_+}$, which by definition means that $\overline\nabla^+\phi=0$ for the metric $b$. We assume that, along $\partial \mathbb H^n_+$, $\phi$ satisfies one of the chirality boundary conditions (\ref{chirbd2}). Thus, 
\begin{equation}\label{chirbd22}
{\overline e_0}\cdot{\overline e_n}\cdot\phi=\pm\phi,
\end{equation} 
where, here and in the next proposition, $\{\overline e_\alpha\}$ is an adapted orthonormal frame with respect to $\overline b$.
As the space of Killing spinors in $\mathbb H^n$ is identified with $\mathbb C^d$,  $d=\big[\frac{n}{2}\big]$, the quadratic form $\mathcal K:\mathbb C^d\to\mathcal N^+_b\oplus\mathfrak K^+_b$ is given by the following proposition:
\begin{proposition}\label{killingder}
Each Killing spinor $\phi$ as above gives rise to an element 
\[
\mathcal K(\phi):=(V_\phi,W_\phi)\in\mathcal N_b^+\oplus\mathfrak K_b^+\cong\mathbb L^{1,n-1}\oplus\mathbb L^{1,n-1}
\] by means of the prescriptions
\begin{equation}\label{killingder2}
V_\phi=\langle \phi,\phi\rangle, \quad W_\phi=\langle {\overline e_0}\cdot {\overline e_i}\cdot\phi,\phi\rangle {\overline e_i}. 
\end{equation}
Moreover, any $V\in \mathcal N_b^+$ or $W\in\mathfrak K_b^+$ on the corresponding future light cone  may be obtained in this way.	
\end{proposition}
\begin{proof}
Define a $1$-form on ${\rm AdS}^{n,1}_+$ by
\[
\alpha_\phi(Z)=\langle{\overline e_0}\cdot Z\cdot \phi,\phi\rangle=(Z\cdot\phi,\phi) 
\]
A simple computation shows that 
\[
\left(\overline\nabla_Z\alpha_\phi\right)(Z')=\frac{\bf i}{2}\left(\left(Z\cdot Z'\cdot-Z'\cdot Z\cdot\right)\phi,\phi\right)=-\left(\overline\nabla_{Z'}\alpha_\phi\right)(Z),
\]
so the dual vector field 
\[
\widetilde W_\phi=\langle {\overline e_0}\cdot {\overline e_\alpha}\cdot\phi,\phi\rangle {\overline e_\alpha}
\] is Killing (with respect to $\overline b$).
Since $\langle{\overline e_0}\cdot{\overline e_0}\cdot \phi,\phi\rangle=V_\phi$, we have 
$\widetilde W_\phi=V_\phi{\overline e_0}+W_\phi$, which we identify to $\mathcal K(\phi)=(V_\phi,W_\phi)$. 
It is easy to check that 
\[ 
dV_\phi(X)={\bf i}\langle X\cdot \phi,\phi\rangle, \quad X\in\Gamma(T\mathbb H^n_+), 
\]
so that, along $\partial\mathbb H^n_+$, $\partial V_\phi/\partial y_n={\bf i}\langle \partial_{y_n}\cdot \phi,\phi\rangle=0$, where in the last step we used the chirality boundary condition. This shows that $V_\phi\in\mathcal N_b^+$. Also, from  Remark \ref{rem_chiral} we get that, along $\partial\mathbb H^n_+$, $W_\phi=\pm V_\phi{\overline e_n}$, which means that  $W_\phi\in\mathfrak K_b^+$.  Finally, the last assertion of the proposition for $V\in\mathcal N_b^+$ is well-known (see \cite[Proposition 5.1]{AdL}) and the corresponding statement for $W\in\mathfrak K^+_b$ follows from the isomomorphism $\mathfrak K^+_b\cong \mathcal N_b^+$ already established in Proposition \ref{propo:K+}. 
\end{proof}

Under the conditions of Theorem \ref{maintheohyp}, the standard analytical argument  
allows us to obtain a spinor $\psi \in \Gamma(\mathbb S_M)$ which is {\rm Killing harmonic} ($\overline{\mathcal D}^+\psi=0$), asymptotes $\phi$ at infinity, and satisfies one of the chirality boundary conditions (\ref{chirbd2}) (with sign agreeing with the one for $\phi$). Note that, on asymptotically hyperbolic manifolds, for given $\eta \in L^2(\mathbb S_M)$ a solution of 
	\[
	\left\{
	\begin{array}{rcll}
	\overline{\mathcal D}^+\varphi & = & \eta & \qquad \text{on } \, M\\
	\mathcal{Q} \varphi & = & \pm\varphi & \qquad \text{on } \, \Sigma
	\end{array}
	\right.
	\]
can be found in $H^1(\mathbb S_M)$ and not in a larger space like the completion  $\mathscr{H}$ appearing in the proof of Proposition \ref{surjdir}. This makes the proof much simpler, in the sense that it follows directly from  the existence theory in \cite{GN}, and relies on the fact that the weight in the Hardy inequality on $\mathbb H^n$ has a positive lower bound, cf. \cite[Theorem 9.10] {bartnikchrusciel}.

Replacing this $\psi$ in (\ref{intpartf}) is the first step in proving the following Witten-type formula, which extends results in \cite{CH,AdL,Ma}:
\begin{theorem}\label{witt-type}
Under the conditions above, there holds 
	\begin{eqnarray}\label{witt2}
\frac{1}{4}\widetilde{\mathcal K}(\phi)
& = & \int_{M}(|\overline\nabla^+\psi|^2+\langle \mathcal W\psi,\psi\rangle)dM\nonumber\\
&  & \quad {+\frac{1}{2}\int_{\Sigma} \left(H_g\pm\pi_{nn}\right)|\psi|^2 d\Sigma,}
\end{eqnarray}
where $\widetilde {\mathcal K}={\mathfrak m}_{(g,h,F)}\circ \mathcal K$.
\end{theorem} 
\begin{proof}
From (\ref{intpartf}) we have
	\begin{eqnarray*}
	\lim_{r\to+\infty}\int_{S^{n-1}_{r,+}}\langle(\overline\nabla^+_{e_i}+e_i\cdot\overline{\mathcal D}^+) \psi,\psi\rangle e_i\righthalfcup dM & = & \int_{M}(|\overline\nabla^+\psi|^2+\langle \mathcal W\psi,\psi\rangle)dM\\
	&  & \quad {+\frac{1}{2}\int_{\Sigma}\left(H_g\pm\pi_{nn}\right)|\psi|^2 d\Sigma.}
	\end{eqnarray*}
Hence, in order to prove (\ref{witt2}) we must check that 
\begin{equation}\label{witt1}
	\lim_{r\to+\infty}\int_{S^{n-1}_{r,+}}\langle(\overline\nabla^+_{e_i}+e_i\cdot\overline{\mathcal D}^+) \psi,\psi\rangle e_i\righthalfcup dM=\frac{1}{4}\mathfrak m_{(g,h,F)}(V_\phi,W_\phi). 
\end{equation}
We note that (\ref{dechyp1}) may be rewritten as 
	\begin{equation}\label{dechyp2}
\overline\nabla^+_{e_i}+e_i\cdot\overline{\mathcal D}^+ = \nabla^+_{e_i}+e_i\cdot{\mathcal D}^+-\frac{1}{2}\pi_{ij}e_0\cdot e_j\cdot,
\end{equation}
where 
\[
\nabla^+_X=\nabla_{X}+\frac{\rm \bf i}{2}X\cdot, \quad {\mathcal D}^+=\mathcal D-\frac{n{\rm \bf i}}{2},
\]
are the intrinsic Killing connection and the intrinsic Killing Dirac operator, respectively. 
Since the computation in \cite[Section 5]{AdL} gives
\[
	\lim_{r\to+\infty}\int_{S^{n-1}_{r,+}}\langle(\nabla^+_{e_i}+e_i\cdot{\mathcal D}^+) \psi,\psi\rangle e_i\righthalfcup dM=\frac{1}{4}\mathfrak m_{(g,h,F)}(V_\phi,0),
\]
and it is clear from (\ref{enerdef2}) and (\ref{killingder2}) that 
\begin{eqnarray*}
	\lim_{r\to+\infty}\frac{1}{2}\int_{S^{n-1}_{r,+}}\pi_{ij}\langle e_0\cdot e_j\cdot \psi,\psi\rangle e_i\righthalfcup dM 
	& = &  -	\lim_{s\to+\infty}\frac{1}{2}\int_{S^{n-1}_{s,+}}\pi(\mu,e_j)\langle e_0\cdot e_j\cdot \phi,\phi\rangle  d\Sigma\\
	& = & 
	-\frac{1}{4}\mathfrak m_{(g,h,F)}(0,W_\phi),
\end{eqnarray*}
we readily obtain (\ref{witt1}). 
\end{proof}

The inequality $\widetilde{\mathcal K}\geq 0$ in Theorem \ref{maintheohyp} is an immediate consequence of (\ref{witt2}) and the assumed DECs. 
Indeed, given any Killing spinor $\phi\in\Gamma(\mathbb S_{\mathbb H^n})$ satisfying \eqref{chirbd22},
by Remark \ref{claim2} and using that $(\varrho\righthalfcup\pi)^{\upvdash}=\pi_{nn}$, the r.h.s. of (\ref{witt2}) is non-negative. 
Using both signs on (\ref{witt2}), this covers the full space $\mathbb C^d$ of Killing spinors on $\mathbb H^n$ and proves that $\widetilde{\mathcal K}\geq 0$. 
Observe that the argument also works when $\overline {\mathcal D}^+$ is replaced by $\overline {\mathcal D}^-$.

As for the rigidity statement, the assumption $\widetilde{\mathcal K}=0$ implies that $\mathbb S_M$ is trivialized by the Killing spinors $\{\psi_m\}$ associated to the basis  $\{\phi_m\}$. From this point on, the argument is pretty much like that in \cite{Ma}, so it is omitted.  As for the remaining properties of $\Sigma$, they are readily checked by combining the arguments in the proofs of \cite[Theorem 5.4]{AdL} and Theorem \ref{main} above. This proves Theorem \ref{maintheohyp}.

Lastly, we prove  Theorem \ref{maintheoohyp2}.
That the inequality $\widetilde 
{\mathcal K}\geq 0$ implies the mentioned causal character of $(\mathcal E,\mathcal P)$ follows from the last statement in  Proposition \ref{killingder}. Also, $\mathcal E=\mathcal P=0$ clearly implies that $\widetilde{\mathcal K}=0$.


\section{Appendix: the Killing development of $M$}

Computationally, the Killing development is quite similar to its Riemannian counterpart, described in Claim 4 of Section 5. Having fixed a local orthonormal frame $\{e_i\}$ on $M$, with dual coframe $\{\theta^j\}$ and connection forms $\omega^i_j$, we consider $\widetilde M = M \times \R$ endowed with the tensor
$$
\widetilde g = - V^2 du^2 + \sum_i (\theta^i - W^i d u) \otimes ( \theta^i - W^i du),
$$
where $W = W^ie_i$ and we implicitly assume $W^i, V, \theta^i, \omega^i_j$ to be pulled-back to $\widetilde M$ via the projection $\mathfrak{p}: \widetilde M \to M$ onto the first factor. Clearly, $\widetilde g$ is a Lorentzian metric on $\{V>0\} \times \R$, in particular on the entire $\widetilde M$ in the assumptions of Theorem \ref{main} for $E=|P|$ (that we showed to imply $V>0$ on $M$). One readily checks the following properties:
\begin{itemize}
\item[-] The frame 
$$
\widetilde e_0 = \frac{1}{V} \big( \partial_u + W^ie_i), \qquad \widetilde e_i = e_i
$$
is orthonormal for $\widetilde g$ ($e_j$ being the field on $\widetilde M$ tangent to fibers $\{u = \mathrm{const}\}$ and projecting to $e_j$), dual to the coframe 
	\[
	\widetilde \theta^0 = V du, \qquad \widetilde \theta^i = \theta^i - W^i du.
	\]
In particular, $\widetilde e_0$ a timelike, unit normal field to each slice $\{u = \mathrm{const}\}$. 
	\item[-] The forms
$$
\left\{ \begin{array}{l}
\widetilde \omega^j_0 = \widetilde \omega^0_j = \frac{V_j}{V} \widetilde \theta^0 + \frac{(\mathscr{L}_W g)_{jk}}{2V} \widetilde \theta^k \\[0.2cm]
\widetilde \omega^i_j = \omega^i_j + \frac{W^i_j - W^j_i}{2V} \widetilde \theta^0
\end{array}\right.
$$
satisfy the structure equations
$$
\left\{ \begin{array}{l}
d \widetilde \theta^\alpha = -\widetilde \omega^\alpha_\beta \wedge \widetilde \theta^\beta \\[0.2cm]
\widetilde \omega^i_j = - \widetilde \omega^j_i, \quad \widetilde \omega^j_0 = \widetilde \omega^0_j, 
\end{array}\right.
$$
and are therefore the connection forms of the Levi-Civita connection $\widetilde \nabla$ of $\widetilde g$. When \eqref{eq_VW} holds,
\begin{equation}\label{conn_form}
\left\{ \begin{array}{l}
\widetilde \omega^j_0 = \widetilde \omega^0_j = \frac{V_j}{V} \widetilde \theta^0 + h_{jk} \widetilde \theta^k = h_{jk} \theta^k\\[0.2cm]
\widetilde \omega^i_j = \omega^i_j.
\end{array}\right.
\end{equation}
\item[-] The second fundamental form of each slice in the $\widetilde e_0$ direction is $h_{ij} \theta^i \otimes \theta^j$, and the field $\partial_u$ is $\widetilde \nabla$-parallel. 
\item[-] The components $\widetilde R_{\alpha\beta\gamma\delta}$ of the curvature tensor
$\widetilde R$ of $\widetilde g$ satisfy 
\begin{equation}\label{curv_Kd}
\begin{array}{lcl}
\widetilde R_{ijkt} & = & \disp R_{ijkt} + h_{ik}h_{jt} - h_{it}h_{jk}, \\[0.3cm]
\widetilde R_{ijk0} & = & \big[ R_{ijkt} + h_{ik}h_{jt}-h_{it}h_{jk}\big] \frac{W^t}{V} = h_{kj,i}-h_{ki,j} \\[0.3cm]
\widetilde R_{i0k0} & = & \disp \big( h_{it,k} - h_{ik,t} \big)\frac{W^t}{V}. 
\end{array}
\end{equation}
\end{itemize}
Here, we use the agreement
	$$
	\begin{array}{l}
	\disp \widetilde R(\widetilde e_\alpha,\widetilde e_\beta)\widetilde e_\gamma = \widetilde \nabla_{\widetilde e_\alpha} \widetilde \nabla_{\widetilde e_\beta} \widetilde e_\gamma - \widetilde \nabla_{\widetilde e_\beta} \widetilde \nabla_{\widetilde e_\alpha} \widetilde e_\gamma - \widetilde \nabla_{[\widetilde e_\alpha, \widetilde e_\beta]}\widetilde e_\gamma = R^\delta_{\gamma \alpha \beta} \widetilde e_\delta \\[0.2cm]
	\widetilde R_{\delta \gamma \alpha\beta} = \widetilde g\big( \widetilde R(\widetilde e_\alpha,\widetilde e_\beta) \widetilde e_\gamma, \widetilde e_\delta \big).
	\end{array}
	$$
In the setting of Theorem \ref{main}, for $E = |P|$, the immersion $\widetilde \Sigma : = \Sigma \times \R \hookrightarrow \widetilde M$ is totally geodesic: indeed, by Claim 2 and \eqref{conn_form},
	\[
	\langle \widetilde \nabla_{\widetilde e_A} \widetilde e_B, \widetilde e_n \rangle = b_{AB} = 0, \quad \langle \widetilde \nabla_{\widetilde e_A} \widetilde e_0, \widetilde e_n \rangle = h_{An} = 0, 
	\]
and 
	\[
	\langle \widetilde \nabla_{\widetilde e_0} \widetilde e_0, \widetilde e_n \rangle = \langle \frac{V_j}{V} \widetilde e_j, \widetilde e_n \rangle = \frac{V_n}{V} = 0.
	\]

\section*{Funding}
This work was supported by the Coordena\c{c}\~ao de Aperfei\c{c}oamento de Pessoal de N\'ivel Superior [88881.169802/2018-01 to S. Almaraz], 
the Conselho Nacional de Desenvolvimento Cient\'ifico e Tecnol\'ogico [309007/2016-0 to S. Almaraz, 312485/2018-2 to L. de Lima], 
the Funda\c{c}\~ao Cearense de Apoio ao Desenvolvimento Cient\'ifico e Tecnol\'ogico [FUNCAP/CNPq/PRONEX 00068.01.00/15 to S. Almaraz and L. de Lima], 
and the Scuola Normale Superiore [SNS17\_B\_MARI to L. Mari].

\noindent \textbf{Acknowledgements.} This work was carried out during the first author's visit to Princeton University in the academic year 2018-2019. He would like to express his deep gratitude to  Prof. Fernando Marques and the Mathematics Department.
The three authors thank the anonymous referees for their interesting comments, which helped us to improve the presentation.


\begin{thebibliography}{9999999} 
	\bibitem[ABdL]{ABdL} S. Almaraz, E. Barbosa, L. L. de Lima,
	A positive mass theorem for asymptotically flat manifolds with a non-compact boundary. {\em Commun. Anal. Geom.} {\bf 24}(4) (2016), 673-715.
	
	\bibitem[AdL]{AdL} S. Almaraz, L. L. de Lima, The mass of an asymptotically hyperbolic manifold with a non-compact boundary, to appear in {\em Ann. Henri Poincar\'e}.
	
	\bibitem[ACG]{ACG} L. Andersson, M. Cai,  G. J. Galloway. Rigidity and positivity of mass for asymptotically hyperbolic manifolds. {\em Ann. Henri Poincar\'e}, {\bf 9}(1) (2008) 1-33.
	
	\bibitem[AD]{AD}
	L. Andersson, M. Dahl. Scalar curvature rigidity for asymptotically locally hyperbolic manifolds. Ann. Global Anal. Geom., {\bf 16}(1) (1998) 1-27.

	\bibitem[AE]{AE} I.G, Avramidi, G. Esposito, Gauge theories on manifolds with boundary. Comm. Math. Phys. {\bf 200} (1999), no. 3, 495-543.  

	\bibitem[AH]{AH} A. Ashtekhar, R.O. Hansen, A unified treatment of null and spatial infinity in general relativity, I. Universal structure, asymptotic symmetries and conserved quantities at spatial infinity {\em J. Math. Phys.} {\bf 19} (1978), no. 7, 1542-1566.
	
	\bibitem[AS]{AS} 	A.F. Astaneh, S.N. Solodukhin,  Holographic calculation of boundary terms in conformal anomaly, {\em Phys.Lett.B} {\bf 769} (2017) 25-33. 
	
	\bibitem[Bar]{Bar} R. Bartnik, 
	The mass of an asymptotically flat manifold. 
	{\em Commun. Pure and Appl. Math.} {\bf 39} (1986), 661-693.

	\bibitem[BaC]{bartnikchrusciel} R. Bartnik and P.T. Chru\'sciel, Boundary value problems for Dirac-type equations, {\em J. Reine Angew. Math.} {\bf 579} (2005), 13-73.

	\bibitem[BC]{BC} R. Beig,  P.T. Chru\'sciel, Killing vectors in asymptotically flat space-times. I. Asymptotically translational Killing vectors and the rigid positive energy theorem, {\em J. Math. Phys.} {\bf 37} (1996), no. 4, 1939-1961.

	\bibitem[BY]{BY} 	J.D. Brown,  J.W. York, Jr., Quasilocal energy and conserved charges derived from the gravitational action,
	{\em Phys. Rev. D} {\bf 47} (1993), 1407-1419. 
	
	\bibitem[Ch]{Ch} X. Chai, 	Positive mass theorem and free boundary minimal surfaces, {\em 	arXiv:1811.06254}.
	
	\bibitem[CHZ]{CHZ} 
	D. Chen, O. Hijazi, X. Zhang, 
	The Dirac-Witten operator on pseudo-Riemannian manifolds. 
	{\em Math. Z.} {\bf 271} (2012), no. 1-2, 357-372.	

	\bibitem[CD]{CD} P, Chru\'sciel, E. Delay, The hyperbolic positive energy theorem, {\em arXiv:1901.05263}.  
	
	\bibitem[CH]{CH} 
	P. T. Chru\'sciel  and M. Herzlich,
	The mass of asymptotically hyperbolic Riemannian manifolds.
	{\em Pacific J. Math.} {\bf{212}} (2003), 231-264.
	
	\bibitem[CM]{CM} P.T. Chru\'sciel, D. Maerten, Killing vectors in asymptotically flat space-times. II. Asymptotically
	translational Killing vectors and the rigid positive energy theorem in higher dimensions, {\em J. Math. Phys.} {\bf 47} (2006), no. 2.
	
	\bibitem[CMG]{CMG} C.-S. Chu, R.-X. Miao, W.-Z. Guo, A new proposal for holographic BCFT {\em J. High Energ. Phys.} (2017) 2017: 89. 
	
	\bibitem[CMT]{CMT}	P.T. Chrus\'ciel, D. Maerten, P. Tod, 
	Rigid upper bounds for the angular momentum and centre of mass of non-singular asymptotically anti-de Sitter space-times, {\em Journal of High Energy Physics}, 11, 084 (2006).
	
	\bibitem[CEM]{CEM} J. Corvino, M. Eichmair, P. Miao, P., Deformation of scalar curvature and volume, {\em Math. Ann.}   {\bf 357} (2) (2013) 551-584. 
	
	\bibitem[dL1]{dL1} L.L. de Lima, A Feynman-Kac formula for differential forms on manifolds with boundary and applications, {\em Pacific Journal of Mathematics} {\bf 292}(1) (2018), 177-201. 
	
	\bibitem[dL2]{dL2} L.L. de Lima, Heat conservation for generalized Dirac Laplacians on manifolds with boundary,
	{\em Annali di Matematica Pura ed Applicata} (1923 -), {\bf 199(3)}, 997-1021.
	
	\bibitem[dLGM]{dLGM} L.L. de Lima, F. Gir\~ao, A. Montalb\'an, The mass in terms of Einstein and Newton, {\em  Classical Quantum Gravity} {\bf 36} (2019), no. 7, 075017, 11 pp. 
	
	\bibitem[D]{D} L. Ding, Positive mass theorems for higher dimensional Lorentzian manifolds, {\em J. of Math. Phys.},  {\bf 49}, 022504 (2008).
		
	\bibitem[Ei]{Ei}	
	M. Eichmair, The Jang equation reduction of the spacetime positive energy theorem in dimensions less than eight, {\em Comm. Math. Phys.} {\bf 319} (2013), no. 3, 575-593. 
		
	\bibitem[EHLS]{EHLS}
	M. Eichmair, L.-H. Huang, D. Lee, and R. Schoen, The spacetime positive mass theorem in dimensions less than eight, {\em J. Eur. Math. Soc. (JEMS)} {\bf 18} (2016), no. 1, 83-121. 
	
	\bibitem[Es]{Es} G. Esposito, {\em Dirac operators and spectral geometry}. Cambridge Lecture Notes in Physics, 12. Cambridge University Press, Cambridge, 1998.  
	
	\bibitem[GHHP]{GHHP} G.W. Gibbons, S.W. Hawking, G.T. Horowitz, M.J. Perry, Positive mass theorems for black holes,
	{\em Comm. Math. Phys.} {\bf 88} (1983) 295-308.  
	
	\bibitem[Gi]{Gi} P.B Gilkey, {\em Asymptotic formulae in spectral geometry.}
	Studies in Advanced Mathematics. Chapman \& Hall/CRC, Boca Raton, FL, 2004.
	
	\bibitem[GN]{GN} N. Grosse, R. Nakad,
	Boundary value problems for noncompact boundaries of Spin$^c$ manifolds and spectral estimates,
	{\em Proc. Lond. Math. Soc.} {\bf 109} (2014), no. 4, 946-974.
	
	\bibitem[HW]{HW}	D. Harlow, J.-Q. Wu, Covariant phase space with boundaries, {\em 	arXiv:1906.08616.} 
	
	\bibitem[HH]{HH} S.W. Hawking, G.T. Horowitz,	The gravitational Hamiltonian, action, entropy and surface terms. 
	{\em Classical Quantum Gravity} {\bf 13} (1996), no. 6, 1487-1498.

	\bibitem[He1]{He1} M. Herzlich, 
	The positive mass theorem for black holes revisited. 
	{\em J. Geom. Phys.} {\bf 26} (1998), no. 1-2, 97-111.

	\bibitem[He2]{He2} M. Herzlich, Computing asymptotic invariants with the Ricci tensor on asymptotically flat and asymptotically hyperbolic manifolds, {\em Ann. Henri Poincar\'e} {\bf 17} (2016), no. 12, 3605-3617.  

	\bibitem[Hi]{hijazi} O. Hijazi, A conformal lower bound for the smallest eigenvalue of the Dirac operator and Killing spinors. {\em Comm. Math. Phys.} {\bf 104} (1986), no. 1, 151-162. 

	\bibitem[Hu]{Hu} L.-H. Huang, On the center of mass in general relativity. In: FifthMS/IP Studies
	in Advanced Mathematics, vol. 51, pp. 575-591. American Mathematical Society, Providence (2012).

	\bibitem[HL]{HL} L.-H. Huang, D. Lee, Equality in the Spacetime Positive Mass Theorem. {\em Commun. Math. Phys.} {\bf 376} (2020) 2379-2407. 
		
	\bibitem[HJM]{HJM} L.-H. Huang, H.C. Jang, D. Martin,  Mass Rigidity for Hyperbolic Manifolds. {\em Commun. Math. Phys.} {\bf 376} (2020) 2329–2349.
	

	\bibitem[IW]{IW} V. Iyer, R.M. Wald, Some properties of the Noether charge and a proposal for dynamical black hole entropy, {\em Phys. Rev. D} {\bf 50}, (1994) 846-864.

	\bibitem[Ko]{Ko} T. Koerber,  The Riemannian Penrose inequality for asymptotically flat manifolds with non-compact boundary, {\em arXiv:1909.13283}.  
	
	\bibitem[KW]{KW} H.-O. Kreiss, J.  Winicour, 
	Geometric boundary data for the gravitational field. Classical Quantum Gravity {\bf 31} (2014), no. 6, 065004, 19 pp.

	\bibitem[LL]{LL}
	D. A. Lee and P. G. LeFloch, The positive mass theorem for manifolds with distributional curvature. {\em Commun. Math. Phys.} {\bf 339} (2015) 99-120.

	\bibitem[Lo]{Lo}
	J. Lohkamp, The higher dimensional positive mass theorem II, {\em arXiv:1612.07505} (2016).
	
	\bibitem[Ma]{Ma} D. Maerten, Positive energy-momentum theorem for AdS-asymptotically hyperbolic manifolds. {\em Ann. Henri Poincar\'e} {\bf 7} (2006), no. 5, 975-1011.
	
	\bibitem[MT]{MT} P. Miao, L.-F. Tam, Evaluation of the ADM mass and center of mass via the Ricci tensor. {\em Proc. Amer. Math. Soc.} {\bf 144} (2016), 753-761. 

	\bibitem[Mi]{Mi} B.~Michel.
	\newblock Geometric invariance of mass-like asymptotic invariants.
	\newblock {\em J. Math. Phys.}, {\bf 52}(5) 052504 (2011).

	\bibitem[M-O]{M-O} M. Min-Oo, Scalar curvature rigidity of asymptotically hyperbolic spin manifolds. {\em Math. Ann.},
	{\bf 285}(4):527-539, 1989.
	
	\bibitem[PT]{PT} T. Parker, D.H. Taubes, 
	On Witten's proof of the positive energy theorem. 
	{\em Comm. Math. Phys.} {\bf 84} (1982), no. 2, 223-238.

	\bibitem[Ra]{Ra} S. Raulot, Green functions for the Dirac operator under local boundary conditions and applications.  {\em Ann. Global Anal. Geom.} {\bf 39} (2011), no. 4, 337-359. 
	
	\bibitem[RS]{RS} O. Reula, O.  Sarbach, The initial-boundary value problem in general relativity. Internat. J. Modern Phys. D {\bf 20} (2011), no. 5, 767-783.

	\bibitem[SY1]{SY1}
	R. Schoen, S.-T. Yau,
	On the proof of the positive mass conjecture in General Relativity. 
	{\em Comm. Math. Phys.} {\bf{65}} (1979), 45-76.
	
	\bibitem[SY2]{SY2} R. Schoen, S.-T. Yau, The energy and the linear momentum of space-times in general relativity, Comm. Math. Phys. 79 (1981),
	no. 1, 47–51.
	
	\bibitem[SY3]{SY3} R. Schoen, S.-T. Yau, Proof of the positive mass theorem. II, Comm. Math. Phys. 79 (1981), no. 2, 231–260.
	
	\bibitem[SY4]{SY4} R. Schoen, S.-T. Yau,
	Positive Scalar curvature and minimal hypersurface singularities, {\em arXiv:1704.05490}.
	
	\bibitem [Sh]{Sh} Keigo Shibuya, Lorentzian positive mass theorem for spacetimes with distributional curvature. {\em arXiv:1803.10387v1}.

	\bibitem[T]{T} T. Takayanagi, Holographic Dual of a Boundary Conformal Field Theory, {\em Phys.Rev.Lett.} {\bf 107} (2011) 101602. 

	\bibitem[Wa]{Wa}
	X. Wang, 
	Mass for asymptotically hyperbolic manifolds. 
	{\em J. Diff. Geom}. {\bf 57} (2001), 273-299.
	
	\bibitem[Wi]{Wi}
	E. Witten,
	A new proof of the positive energy theorem. 
	{\em Comm. Math. Phys.} {\bf{80}} (1981), 381-402. 
	
	\bibitem[XD]{XD} X. Xu, L. Ding,  
	Positive mass theorems for high-dimensional spacetimes with black holes. {\em Sci. China Math.} {\bf{54}} (2011), no. 7, 1389-1402.
	
\end{thebibliography}
\end{document}